\documentclass[11pt]{amsart}
\usepackage[margin=1.05in]{geometry}
\usepackage{amsmath,amssymb,amsthm,mathtools}
\usepackage{enumitem}
\usepackage{xcolor}
\usepackage{microtype}
\usepackage[colorlinks=true,linkcolor=blue,citecolor=blue,urlcolor=blue]{hyperref}

\newtheorem{theorem}{Theorem}[section]
\newtheorem{lemma}[theorem]{Lemma}
\newtheorem{proposition}[theorem]{Proposition}
\newtheorem{corollary}[theorem]{Corollary}
\theoremstyle{remark}
\newtheorem{remark}[theorem]{Remark}
\newtheorem*{remark*}{Remark}
\theoremstyle{definition}
\newtheorem{definition}[theorem]{Definition}

\newcommand{\eps}{\varepsilon}
\newcommand{\Z}{\mathbb{Z}}
\newcommand{\N}{\mathbb{N}}
\newcommand{\R}{\mathbb{R}}
\newcommand{\E}{\mathbb{E}}
\newcommand{\PP}{\mathbb{P}}
\newcommand{\Hh}{\mathbb{H}}
\newcommand{\I}{\mathbb{I}}
\newcommand{\bn}{\mathbf{n}}
\newcommand{\bX}{\mathbf{X}}
\newcommand{\bY}{\mathbf{Y}}

\newcommand{\cP}{\mathcal{P}}
\newcommand{\cA}{\mathcal{A}}
\newcommand{\fh}{\mathfrak{h}}

\newcommand{\expk}[1]{\exp_{#1}}
\newcommand{\logk}[1]{\log_{#1}}
\newcommand{\one}{\mathbf{1}}
\newcommand{\vt}{\vartheta}
\renewcommand{\Re}{\operatorname{Re}}
\DeclareMathOperator{\TV}{TV}
\DeclareMathOperator{\KL}{KL}

\title[Effective logarithmic two-point Chowla bounds]{Effective logarithmic two-point Chowla bounds in every window}
\author{Scott D.~Hughes}
\date{}

\hypersetup{pdftitle={Effective logarithmic two-point Chowla bounds in every window},pdfauthor={Scott D. Hughes}}

\begin{document}
\begin{abstract}
We give an effective form of the logarithmically averaged two-point Chowla theorem, uniform in the affine forms and the averaging window. For admissible forms $L_i(n)=a_in+b_i$, put $\fh:=\max(a_1,a_2,|b_1|,|b_2|,|a_1b_2-a_2b_1|,2)$. For every $g_1,g_2\in\{\lambda,\mu\}$ and $0<\eps\le1/10$, the correlation $\sum_{x/\omega<n\le x}g_1(L_1(n))g_2(L_2(n))/n$ has absolute value at most $\eps\log\omega$ whenever $x\ge\omega\ge\expk4(C\fh\eps^{-2})$. Here $\expk4$ denotes four iterated exponentials. A translation lemma separates the common translation from the tower: for $\lambda(n+b)\lambda(n+b+k)$, $k\ne0$, cancellation is uniform when $|k|=o(\logk4\omega)$ and $|b|=\omega^{o(1)}$. The same thresholds give logarithmic equidistribution of the four Liouville sign pairs. The proof follows Tao's entropy argument, using two zero-padded blocks in the original affine coordinates, an entropy-to-expectation inequality, and estimates of Matom\"aki--Radziwi\l\l--Tao and Green--Tao. The constant $C$ is effective but depends on uncomputed constants in these inputs.

Separately, for the non-pretentiousness level $N_\lambda(x)$ in Tao's hypothesis, we prove the fully explicit bounds $\ell/3-\frac43\log\ell-7\le N_\lambda(x)\le\ell-\log\ell+2.53/\ell$ for $\ell=\log\log x\ge64$. Under GRH the lower bound improves to $\ell-\log\ell-\log2-2/\ell$, so $N_\lambda(x)=\ell-\log\ell+O(1)$. These bounds use an explicit zero-free region, Harnack's inequality, a finite resonator, and a comparison of smoothed Euler products. The correlation results concern logarithmic averages and do not establish two-point Chowla at natural density.
\end{abstract}
\maketitle

\section{Introduction}

Let $\lambda$ be the Liouville function and $\mu$ the M\"obius function. Tao \cite{Tao16} proved the logarithmically averaged two-point Elliott conjecture; for $\lambda$ it states that for an admissible pair of affine forms $L_i(n)=a_in+b_i$ ($a_i\in\N$, $b_i\in\Z$, $\Delta:=a_1b_2-a_2b_1\neq0$) and every $\eps>0$ there is $A=A(\eps,a_1,a_2,b_1,b_2)$ such that for all $x\ge\omega\ge A$,
\begin{equation}\label{eq:tao}
\Bigl|\sum_{x/\omega<n\le x}\frac{\lambda(L_1(n))\lambda(L_2(n))}{n}\Bigr|\le\eps\log\omega .
\end{equation}
This holds in every window, but Tao's Remark 1.4 states that the threshold $A$, though effective in principle, was not computed because the bounds were expected to be poor. In this note we give an effective threshold, explicit in terms of the constants in two analytic inputs (which are effective but have not been computed). The bound is poor---a four-fold iterated exponential---but its dependence on the forms is linear inside the tower and there is no constraint between $\eps$ and the forms, so one obtains a single estimate uniform over all admissible pairs of bounded height and over every window; and the common translation of the forms can be moved out of the tower altogether (Theorem~\ref{thm:C}). Throughout, $g(n):=0$ for $n\le0$.

\begin{definition}\label{def:height}
The \emph{augmented height} of an admissible pair is $\fh:=\max(a_1,a_2,|b_1|,|b_2|,|\Delta|,2)$; it includes the determinant and is therefore not the usual maximum of the four coefficients, but we say ``height'' for short. Its \emph{block parameter} is
\[
R_0:=\max\Bigl(1,\frac{|b_1|}{a_1},\frac{|b_2|}{a_2}\Bigr),\qquad K:=\lceil4R_0\rceil,\qquad B:=K(a_1+a_2),
\]
and we write $a:=a_1a_2$.
We write $\logk k$, $\expk k$ for the $k$-fold iterated logarithm and exponential, and $p$ always denotes a prime.
\end{definition}

\begin{theorem}\label{thm:A}
There is an effective absolute constant $C_1$ such that for every admissible pair, every $g_1,g_2\in\{\lambda,\mu\}$, every $0<\eps\le1/10$, and all $x\ge\omega\ge\expk4(C_1B\eps^{-2})$,
\begin{equation}\label{eq:main}
\Bigl|\sum_{x/\omega<n\le x}\frac{g_1(L_1(n))g_2(L_2(n))}{n}\Bigr|\le\eps\log\omega .
\end{equation}
Since $B\le15\fh$ (Lemma~\ref{lem:B}), this holds in particular for $\omega\ge\expk4(15C_1\fh\eps^{-2})$; and there is an effective absolute $C$ such that for $x\ge\omega\ge\expk4(C\fh)$, uniformly over admissible pairs of height at most $\fh$,
\begin{equation}\label{eq:rate}
\Bigl|\sum_{x/\omega<n\le x}\frac{g_1(L_1(n))g_2(L_2(n))}{n}\Bigr|\le C\Bigl(\frac{\fh}{\logk4\omega}\Bigr)^{1/2}\log\omega .
\end{equation}
\end{theorem}

The constant $C_1$ depends on the implied constant in the Matom\"aki--Radziwi\l\l--Tao estimate \cite[Theorem 1.3]{MRT15} and on the two implied constants and the admissibility threshold in the Green--Tao restriction estimate \cite[Propositions 3.1(i), 4.2]{GT06}; these are effective but have not been computed, and we do not compute them. The parameter choices below are explicit in terms of these input constants; the final absolute tower constant is not numerically evaluated.

The same threshold controls the frequencies of the four sign pairs (Corollary~\ref{cor:signs}), and the common translation of the forms can be separated from their separation:

\begin{theorem}\label{thm:C}
For an admissible pair put $\beta_i:=b_i/a_i$, $d:=|\beta_1-\beta_2|=|\Delta|/(a_1a_2)$, $K_s:=\lceil4\max(1,d/2)\rceil$, $B_s:=K_s(a_1+a_2)$, $\fh_s:=\max(a_1,a_2,|\Delta|,2)$, and let $m$ be an integer nearest to $(\beta_1+\beta_2)/2$. Then $B_s\le8\fh_s$, and for every $g_1,g_2\in\{\lambda,\mu\}$, every $0<\eps\le1/10$, and all
\begin{equation}\label{eq:threshC}
x\ge\omega\ge\max\bigl\{\expk4((6C_1+1)B_s\eps^{-2}),\ 8\,[4(|m|+1)]^{2/\eps-1}\bigr\},
\end{equation}
the bound \eqref{eq:main} holds, and so does its analogue for the sign-pair sums of Corollary~\ref{cor:signs}. Consequently there is an effective absolute $C$ such that for all $x\ge\omega$ with $\logk4\omega\ge1$,
\begin{equation}\label{eq:rateC}
\Bigl|\sum_{x/\omega<n\le x}\frac{g_1(L_1(n))g_2(L_2(n))}{n}\Bigr|\le C\Bigl[\Bigl(\frac{\fh_s}{\logk4\omega}\Bigr)^{1/2}\log\omega+\log(2+|m|)\Bigr].
\end{equation}
\end{theorem}

For $\lambda(n+b)\lambda(n+b+h)$, $h\ne0$, one has $\fh_s=\max(|h|,2)$ and $m$ is the integer nearest $b+h/2$, so cancellation is uniform as soon as $|h|=o(\logk4\omega)$ and $|b|=\omega^{o(1)}$, rather than $|b|=o(\logk4\omega)$. The proof (\S\ref{sec:centred}) applies Theorem~\ref{thm:A} to the centred forms $L_i(n-m)$ and then translates by $m$ with an elementary lemma; it does not re-run the entropy argument.

\subsection*{What is new and what is not}
Tao proved the qualitative every-window statement for fixed forms. Here we track the \emph{effective joint dependence} on $\eps$, the forms and $\omega$, with linear dependence on the height inside the tower. The proof follows his entropy argument \cite{Tao16}, with five choices that give this dependence. (i) We observe the two forms in their own coordinates, as two blocks $(g_1(a_1\bn+r))_{r\le a_1KH}$ and $(g_2(a_2\bn+r))_{r\le a_2KH}$, instead of normalising both to slope $a_1a_2$; the entropy budget is then $BH\log3$ and the medium-prime band and information target no longer depend on the forms. (ii) The passage from mutual information to a decoupled expectation, which in \cite{Tao16} goes through a ``good'' set, a Markov inequality and an exceptional set, is a single inequality $|\E_PF-\E_QF|\le\sqrt{2V\KL(P\|Q)}$ (Lemma~\ref{lem:ent}). (iii) The Fourier step is run on zero-padded blocks of twice the observed length, with the $j$-sum truncated from the outset, which removes the boundary terms that would otherwise force $\eps\ll1/\fh$. (iv) Tao's affine-invariance lemma carries an $o(1)$ error accumulated $\gg H^2$ times; a window reduction to $x/\omega\ge\omega^{\eps/10}$ (Lemma~\ref{lem:window}) makes it harmless and makes $\bn$ exponentially equidistributed modulo the medium-prime product. (v) The restriction step counts the joint resonance group of the two blocks (Lemma~\ref{lem:restr}), and a translation lemma (Lemma~\ref{lem:trans}) moves the common translation of the forms out of the tower.

The statement is logarithmic: it implies nothing about $\sum_{n\le x}\lambda(n)\lambda(n+h)$ at natural density and nothing about prime tuples. Quantitative every-window estimates are already known for fixed forms. In the basic case $(n,n+1)$, Helfgott and Ubis \cite[Theorem 5.1]{HU19} make Tao's entropy method quantitative, obtaining an error $\ll\log\omega/\min(\logk3\omega,\logk4x)^{1/5}$; Helfgott and Radziwi\l\l \cite[Corollary 1.5]{HR22} obtain the stronger bound $\ll\log\omega/(\log\log\omega)^{1/2}$ by expansion. Their note at the end of \S8.3 extends this conclusion to fixed reduced admissible affine pairs; that note does not state an explicit bound for the dependence on the coefficients. Pilatte \cite[Theorem 1.1]{Pil23} proves a power-of-log saving for the corresponding logarithmic prefix sum. Our contribution is the explicitly tracked joint dependence on $\eps$, the affine forms and every averaging window, together with the separation of the common translation; we claim no rate improvement for fixed forms.

Tao and Ter\"av\"ainen \cite[Theorem 3.1, Remarks 3.2]{TT25} give a quantitative correlation theorem for multiplicative functions under specified equidistribution or non-pretentiousness hypotheses. Its Liouville specialization allows distinct positive slopes and coefficients up to a small absolute power of $\log N$, with a power-of-log saving in natural averages outside a set of scales of small logarithmic density. Logarithmic integration gives corresponding prefix bounds. Guo \cite[Theorem 1.1]{Guo26v2} states $\sup_{1\le h\le(\log x)^A}|\sum_{n\le x}\lambda(n)\lambda(n+h)/n|\ll_A(\log x)^{1-c}$ for every fixed $A>0$, with no exceptional shifts, and \cite[Theorem 1.1]{Guo26v4} states a bound uniform in the terminal point for $1\le h\le x$ outside a single sparse set of shifts; we cite these as stated theorems of a preprint. These coefficient-uniform prefix estimates have much stronger rates and coefficient ranges than ours, but do not directly give an error controlled by $\omega$ alone in every window: subtracting prefix bounds leaves an error of size $(\log x)^{1-c}$, which gives no cancellation when $\log\omega\le(\log x)^{1-c}$. We do not know how to combine their coefficient uniformity with an every-window estimate of the kind considered here.

\subsection*{The non-pretentiousness level}
Tao's theorem is stated for general $1$-bounded multiplicative $g_1,g_2$ under the hypothesis that $g_1$ is non-pretentious up to level $A$ at scale $x$:
\begin{equation}\label{eq:np}
\sum_{p\le x}\frac{1-\Re\bigl(g_1(p)\overline{\chi(p)}p^{-it}\bigr)}{p}\ge A\qquad\text{for all $\chi\bmod q$, $q\le A$, and all $|t|\le Ax$}.
\end{equation}
For $g_1\in\{\lambda,\mu\}$ this hypothesis is not needed in Theorem~\ref{thm:A}: the only place it enters Tao's proof is the short-interval exponential-sum estimate, which for $\lambda$ is available unconditionally \cite[Theorem 1.3]{MRT15} and for $\mu$ follows from it (Lemma~\ref{lem:mrt}). It is nevertheless natural to ask up to what level \eqref{eq:np} holds for $\lambda$. Let $\cA(x)$ be the set of $A\ge1$ for which \eqref{eq:np} holds with $g_1=\lambda$, and $N_\lambda(x):=\sup(\{0\}\cup\cA(x))$. (The set $\cA(x)$ can be empty: at $x=286$ one has $\sum_{p\le286}(1+\cos(14\log p))/p=0.791\ldots<1$, and $14\le286A$ for every $A\ge1$.)

\begin{theorem}\label{thm:N}
Let $\ell:=\log\log x$.
\begin{enumerate}[leftmargin=1.8em,label=\textup{(\roman*)}]
\item For $x\ge286$, $N_\lambda(x)\le\ell+\tfrac12$ (more precisely $\ell+B_M+O(1/\log^2x)$, $B_M=0.2614\ldots$ Mertens' constant; for $\ell\ge64$ this is superseded by (ii)).
\item For $\ell\ge64$,
\[
\frac\ell3-\frac43\log\ell-7\ \le\ N_\lambda(x)\ \le\ \ell-\log\ell+\frac{2.53}\ell.
\]
(The unrounded unconditional lower bound is \eqref{eq:N6}.)
\item Assume GRH for Dirichlet $L$-functions. For $\ell\ge64$,
\[
\ell-\log\ell-\log2-\frac2\ell
\ \le\ N_\lambda(x)\ \le\ \ell-\log\ell+\frac{2.53}\ell.
\]
In particular, $N_\lambda(x)=\ell-\log\ell+O(1)$ under GRH.
\end{enumerate}
\end{theorem}

Part (ii) is fully explicit and has the leading constant $1/3$ of the Vinogradov--Korobov exponent, which is also what one obtains asymptotically from \cite[(1-12)]{MRT15}; the proof bounds $L$ from below on a zero-free disk by Harnack's inequality, using Khale's explicit zero-free region \cite{Kha24} for $|t|\ge12$ and McCurley's \cite{McC84} for $|t|<12$, and never passes through $L(s,\chi^2)$. Part (i) corrects the naive bound $2\ell$ from testing only $t=0$. The upper bound in (ii) averages against $|R(t)|^2$ for a finite resonator built from $\lambda$ on a prime-dependent exponent box; its diagonal terms $n=pm$ are negative. Part (iii) compares two smoothed Euler-product truncations, at cutoffs $x$ and $Y=(\log\mathcal Q_*)^2$ with $\mathcal Q_*:=\ell(x\ell+3)$, before discarding prime powers, so that the prime Mertens constant cancels and the transition bands cost $O(1/\ell)$. In particular, writing $E_\lambda(x):=N_\lambda(x)-\ell+\log\ell$, under GRH
\[
-\log2\le\liminf_{x\to\infty}E_\lambda(x)\le\limsup_{x\to\infty}E_\lambda(x)\le0;
\]
no existence or value of a limiting constant is asserted. The proofs are in Appendix~\ref{app:N}.

\subsection*{Notation}
$e(u):=\exp(2\pi iu)$, $\one_E$ is an indicator, $\pi(y)$ counts primes up to $y$, $\theta(y):=\sum_{p\le y}\log p$, and $\Lambda$ is the von Mangoldt function. Constants $C,c$ are absolute and effective; numbered constants are fixed when introduced. $O^*(E)$ denotes a quantity of absolute value at most $E$. From Rosser--Schoenfeld \cite{RS62} we use: $\pi(y)<1.26\,y/\log y$ for $y>1$; $\pi(2y)-\pi(y)>\frac35y/\log y$ for $y\ge20.5$; $\theta(y)<1.02\,y$ for $y>0$; $\sum_{p\le y}(\log p)/p<\log y$ for $y>1$; $\sum_{p\le y}1/p>\log\log y+B_M-1/(2\log^2y)$ for $y>1$ and $<\log\log y+B_M+1/(2\log^2y)$ for $y\ge286$.

\section{Reductions and parameters}\label{sec:red}

\begin{lemma}\label{lem:B}
$B\le15\fh$, $a_i,|b_i|\le B$, and $|\Delta|\le B^2$.
\end{lemma}
\begin{proof}
$K\le4R_0+1\le5R_0$. If $R_0=1$ then $B\le5(a_1+a_2)\le10\fh$. If $R_0=|b_1|/a_1$ then $R_0(a_1+a_2)=|b_1|+a_2|b_1|/a_1\le|b_1|+|b_2|+|\Delta|/a_1\le3\fh$, using $a_2b_1=a_1b_2-\Delta$; so $B\le15\fh$. The case $R_0=|b_2|/a_2$ is symmetric. Also $a_i\le B$, $|b_i|\le a_iR_0\le B$, and $|\Delta|\le|a_1b_2|+|a_2b_1|\le2aR_0\le K^2(a_1+a_2)^2=B^2$.
\end{proof}

\subsection{Window reduction}
\begin{lemma}\label{lem:window}
Let $0<\eps\le1/10$, $\omega\ge e^{10/\eps}$, $x\ge\omega$, $|f|\le1$. Put $T:=\omega^{\eps/10}$ and $\omega':=\min(\omega,x/T)$. Then $\omega^{1/2}\le\omega'\le\omega$, $x/\omega'\ge\omega'^{\,\eps/10}$, and
\[
\Bigl|\sum_{x/\omega<n\le x}\frac{f(n)}n\Bigr|\le\Bigl|\sum_{x/\omega'<n\le x}\frac{f(n)}n\Bigr|+\frac\eps5\log\omega .
\]
\end{lemma}
\begin{proof}
The sums differ by the terms $x/\omega<n\le\max(x/\omega,T)$, of total absolute value $\le\sum_{n\le T}1/n\le\log T+1\le\frac\eps5\log\omega$ as $\log\omega\ge10/\eps$. If $\omega'=\omega$ then $x/\omega\ge T=\omega^{\eps/10}$. If $\omega'=x/T<\omega$ then $x/\omega'=T\ge\omega'^{\,\eps/10}$ and $\omega'=x/T\ge\omega^{1-\eps/10}\ge\omega^{1/2}$.
\end{proof}

We prove the core estimate at tolerance $0<\eta\le1/10$, assuming $\omega\ge A_0(\eta)$ and $x/\omega\ge\omega^{\eta/20}$, with $A_0(\eta)$ defined below. To deduce Theorem~\ref{thm:A} at tolerance $\eps$, apply Lemma~\ref{lem:window} with $\eta=\eps/2$. The parameter choices below ensure $A_0(\eta)\ge e^{10/\eta}$ and $A_0(\eta)^{\eta/10}\ge\fh$. Thus, if $\omega\ge A_0(\eps/2)^2$, the retained ratio satisfies $\omega'\ge\sqrt\omega\ge A_0(\eta)$ and
\[
x/\omega'\ge(\omega')^{\eta/10}\ge(\omega')^{\eta/20}.
\]
The retained sum and the discarded terms therefore have total absolute value at most
\[
\eta\log\omega'+\frac{\eta}{5}\log\omega\le\frac{3\eps}{5}\log\omega\le\eps\log\omega.
\]
Any index with $L_i(n)\le0$ satisfies $n\le\fh\le T$ and is absent from the retained window. Throughout \S\S2--6 we write $\eps$ for the core tolerance and assume
\begin{equation}\label{eq:standing}
x/\omega\ge\omega^{\eps/20}.
\end{equation}
The parameter choices ensure $L_i(n)\ge1$ for every $n$ in this window. Under \eqref{eq:standing}, for $\omega\ge A_0$,
\begin{equation}\label{eq:mass}
W:=\sum_{x/\omega<n\le x}\frac1n=\log\omega+O^*(2\omega/x),\qquad \frac\omega x\le\omega^{-\eps/20}\le A_0^{-\eps/20},\qquad\log\frac x\omega\ge\frac{\eps}{20+\eps}\log x .
\end{equation}

\subsection{Parameters}\label{sec:param}
Let $C_{\rm L}\ge1$ be the implied constant in \cite[Theorem 1.3]{MRT15}, $C_*:=10C_{\rm L}$ (Lemma~\ref{lem:mrt}), let $C_{\rm GT}\ge1$, $0<c_{\rm GT}\le1$ be the implied constants in \cite[Prop.\ 4.2, Prop.\ 3.1(i)]{GT06} for $k=1$, $F(n)=n$, exponent $4$, $R=\lfloor N^{1/10}\rfloor$, and let $N_{\rm GT}\ge1$ be an effective threshold such that both propositions hold for all $N\ge N_{\rm GT}$. Put $\vt:=1/100$, and define
\begin{align}
\delta&:=\frac{\eps^2\vt}{1440^2},\qquad L:=\log\frac{9\vt}{\delta}=\log\frac{9\cdot1440^2}{\eps^2},\qquad S:=\frac{2B\log3}{\delta},\label{eq:delta}\\
C_5&:=102400\cdot480^4\,C_{\rm GT}^4c_{\rm GT}^{-2},\notag\\
B_0&:=\max\Bigl\{360\,C_*C_5\,a^5|\Delta|K^2\vt^{-2}\eps^{-5},\ \ 10^{12}(B+\fh+1)^4\eps^{-2},\ \ \log(2N_{\rm GT})\Bigr\},\label{eq:Hminus}\\
H_-&:=\bigl\lceil\exp(2B_0\log(2B_0))\bigr\rceil,\notag\\
H_+&:=\exp\Bigl(H_-+2\exp\bigl((\log\log H_-+L)\,e^{S}\bigr)\Bigr),\qquad U:=BH_+,\qquad A_0:=\exp(U^4),\qquad Q:=4U.\label{eq:Hplus}
\end{align}
Since $v/\log v>B'$ for $v=2B'\log(2B')$, $B'\ge1$ (as $\log v<2\log(2B')$),
\begin{equation}\label{eq:Hminusprop}
\frac{\log H_-}{\log\log H_-}> B_0 .
\end{equation}
Sizes: $a\le B^2$, $K\le B$, and Lemma~\ref{lem:B} gives $|\Delta|,\fh\le B^2$. Thus $B_0\le CB^{14}\eps^{-5}$ and $\log H_-\le CB^{14}\eps^{-5}\log(B/\eps)$ (with $C_*,C_5,N_{\rm GT}$ fixed), so $\log\log H_-+L=O(\log(B/\eps))$; and since $H_-+2e^{Z}\le3e^{\max(\log H_-,Z)}$,
\begin{equation}\label{eq:Asize}
\begin{aligned}
\logk3H_+&\le\log2+\max\bigl\{\log\log H_-,\ S+\log(\log\log H_-+L)\bigr\}\le C_7B\eps^{-2},\\
\logk4A_0&\le\log4+\log B+\logk3H_+\le C_8B\eps^{-2}.
\end{aligned}
\end{equation}
We record: $H_+\ge\exp(H_-)$; $H_-\ge B_0^2\ge(10^{6}(B+\fh+1)/\eps)^4$ and $\log H_-\ge100$; $U\ge H_+\ge H_-\ge\fh^2$; and
\begin{equation}\label{eq:Abig}
\log A_0=U^4\ge\frac{40}{\eps}\bigl(2\vt U+U+\log(8U)\bigr),\qquad U\ge400\log(8U),\qquad \frac{\eps U^4}{20}\ge\log(8000)+4\log U .
\end{equation}
Theorem~\ref{thm:A} follows from the core statement with $A=A_0(\eps/2)^2$, and \eqref{eq:Asize} gives $\logk4A\le\log2+4C_8B\eps^{-2}\le C_1B\eps^{-2}$. The rate form \eqref{eq:rate} follows by choosing $\eps:=(15C_1\fh/\logk4\omega)^{1/2}\le1/10$ when $\logk4\omega\ge1500C_1\fh$.

\subsection{The probabilistic set-up}
Assume \eqref{eq:standing}, $x\ge\omega\ge A_0$, and, for contradiction,
\begin{equation}\label{eq:contra}
\Bigl|\sum_{x/\omega<n\le x}\frac{g_1(L_1(n))g_2(L_2(n))}n\Bigr|>\eps\log\omega .
\end{equation}
Let $\bn$ be the random integer in $(x/\omega,x]$ with $\PP(\bn=n)=1/(nW)$, and $C_0:=\E\,g_1(L_1(\bn))g_2(L_2(\bn))$. By \eqref{eq:mass},
\begin{equation}\label{eq:C0}
|C_0|\ge\eps\,\frac{\log\omega}{W}\ge\frac\eps2 .
\end{equation}

\begin{lemma}[Approximate affine invariance]\label{lem:aff}
Let $E_{\rm aff}:=\dfrac{2\log(2Q)}{\log\omega}+\dfrac{10\,Q\omega}x$. Then $E_{\rm aff}\le\dfrac1{100U^3}$. For integers $1\le q\le Q$, $|r|\le Q$, and any $X:\Z\to\mathbb C$ with $|X|\le M_X$,
\[
\Bigl|\E\bigl[X(\bn)\one_{\bn\equiv r\,(q)}\bigr]-\frac1q\E\,X(q\bn+r)\Bigr|\le M_XE_{\rm aff},
\]
and $\TV(\bn,\bn+r)\le E_{\rm aff}$. In particular, for $q\le Q$ and any residue $r$ (taken in $[0,q)$), $\PP(\bn\equiv r\ (q))\le1/q+E_{\rm aff}$.
\end{lemma}
\begin{proof}
As in \cite[Lemma 2.5]{Tao16}. Write $n=qm'+r$. By \eqref{eq:mass} and \eqref{eq:Abig}, $|r|\le Q\le x/(2\omega)$, so $qm'\ge n-|r|\ge x/(2\omega)$ and $1/n=(1/qm')(1+O^*(4Q\omega/x))$. The range $x/\omega<qm'+r\le x$ differs from $x/\omega<m'\le x$ by two end ranges of $1/m'$-mass $\le\log q+2Q\omega/x$ each. Dividing by $W\ge\log\omega-1$ gives the first claim; the total-variation bound is the case $q=1$ tested on indicator functions; the last claim is the first with $X\equiv1$. Numerically, $\log\omega\ge U^4$ and $\omega/x\le e^{-\eps U^4/20}$, so $E_{\rm aff}\le2\log(8U)U^{-4}+40Ue^{-\eps U^4/20}\le1/(100U^3)$ by \eqref{eq:Abig}.
\end{proof}

\noindent\textbf{Accounting for $E_{\rm aff}$.} The lemma is applied with $q\in\{1,p\}$ and $|r|\le\max(2KH_+,\,|b_i|)\le Q$ ($|b_i|\le B\le U$); in Lemma~\ref{lem:mrt} with $q=c\le B$, $r=d$, $|d|\le B$, and with $q=1$, $1\le r\le H_-$; and, for M\"obius, with $q=p$ and the residue of $-b_i\bar a_i$ modulo $p$. The multiplicities are stated at each use; every accumulated error is bounded there, and none exceeds $\eps KH/(1000\log H)$ (Proposition~\ref{prop:main}) or $3$ in entropy (Lemma~\ref{lem:subadd}).

\subsection{Medium primes and elementary conditions}\label{sec:Hminus}
For $H\ge H_-$ let $\cP_H$ be the set of primes in $(\vt H/2,\vt H]$ and $P_H:=\prod_{p\in\cP_H}p$. From $H_-\ge(10^6(B+\fh+1)/\eps)^4$ and $\log H_-\ge100$ one checks, for all $H\ge H_-$:
\begin{gather}
|\cP_H|\le\frac{2\vt H}{\log H},\qquad\frac1{4\log H}\le\sum_{p\in\cP_H}\frac1p\le\frac4{\log H},\qquad\log P_H\le1.02\,\vt H,\label{eq:PH}\\
\frac{\vt H}2>\max\bigl(a_1,a_2,(2aKH)^{1/10}\bigr),\qquad 12\log H\le\delta H,\qquad H^{1/20}\ge2,\label{eq:Hcond}\\
H\ge\frac{3200}{\eps\vt},\qquad 2aKH\ge N_{\rm GT}.\notag
\end{gather}
(For the prime count, $\log(\vt H)\ge0.63\log H$ as $H^{0.37}\ge100$, and $\pi(\vt H)-\pi(\vt H/2)\ge\frac{3}{5}\cdot\frac{\vt H/2}{\log(\vt H/2)}$ as $\vt H/2\ge20.5$; for the cutoff, $2aK\le B^2$ and $H^{0.9}>200B^{1/5}$; for the entropy condition, $\log H\le\sqrt H$ and $\sqrt H\ge12/\delta$.)

\section{The short-interval input}\label{sec:mrt}

\begin{lemma}\label{lem:mrt}
Let $u_-:=3C_*\,\dfrac{\log\log H_-}{\log H_-}$ with $C_*=10C_{\rm L}$. For $g\in\{\lambda,\mu\}$, all integers $J,c,d$ with $H_-\le J\le U$, $1\le c\le B$, $|d|\le B$, and all $\alpha\in\R$,
\begin{equation}\label{eq:2.8}
\E\Bigl|\frac1J\sum_{j=1}^Jg(\bn+j)e(\alpha j)\Bigr|\le u_-,\qquad
\E\Bigl|\frac1J\sum_{j=1}^Jg(c\bn+d+j)e(\alpha j)\Bigr|\le2c\,u_-,\qquad
\bigl|\E\,g(c\bn+d)\bigr|\le2c\,u_- .
\end{equation}
\end{lemma}
\begin{proof}
\emph{Step 1: the integral bound.} By \cite[Theorem 1.3]{MRT15}, for $10\le J\le T$,
\begin{equation}\label{eq:mrt}
\sup_\alpha\int_0^T\Bigl|\sum_{y\le n\le y+J}\lambda(n)e(\alpha n)\Bigr|dy\le C_{\rm L}\,\rho(J,T)\,JT,\qquad\rho(J,T):=\frac{\log\log J}{\log J}+(\log T)^{-1/700}.
\end{equation}
We claim the same holds for $\mu$ with $C_*=10C_{\rm L}$ in place of $C_{\rm L}$ when $J\ge e^{100}$ and $T\ge J^2$. Write $\mu(n)=\sum_{d^2\mid n}\mu(d)\lambda(n/d^2)$ and $D:=\lfloor J^{1/4}\rfloor$. For $d\le D$, the terms with $d^2\mid n$ in $[y,y+J]$ are $\mu(d)\sum_{y/d^2\le v\le(y+J)/d^2}\lambda(v)e(\alpha d^2v)$, a sum over a window of length $J/d^2\in[J^{1/2},J]$; substituting $y=d^2y'$ and applying \eqref{eq:mrt} at length $\lceil J/d^2\rceil$ and scale $T/d^2$ (allowing one endpoint term, and using $\log\log(J/d^2)/\log(J/d^2)\le2\log\log J/\log J$ and $\log(T/d^2)\ge\frac34\log T$), the contribution of $d\le D$ is at most $\sum_{d\le D}\bigl(4C_{\rm L}\rho JT/d^2+T\bigr)\le8C_{\rm L}\rho JT+DT$. For $d>D$, count multiples of $d^2$ before integrating: $\int_0^T\#\{n\in[y,y+J]:d^2\mid n\}dy\le J(T+J)/d^2$, so the total is $\le2JT\sum_{d>D}d^{-2}\le4TJ^{3/4}$. Hence the $\mu$ integral is $\le JT\bigl(8C_{\rm L}\rho+J^{-3/4}+4J^{-1/4}\bigr)\le10C_{\rm L}\rho JT$, as $\rho\ge\log\log J/\log J\ge4J^{-1/4}+J^{-3/4}$ for $J\ge e^{100}$.

\emph{Step 2: the logarithmic average.} Put $F(n):=|J^{-1}\sum_{j=1}^Jg(n+j)e(\alpha j)|$ and $S(T):=\sum_{n\le T}F(n)$. For $y\in(n-1,n)$ the window $[y,y+J]$ contains exactly the integers $n,\dots,n+J-1$, so $|\sum_{y\le m\le y+J}g(m)e(\alpha m)|=JF(n-1)$; hence $S(T)\le J^{-1}\int_0^{T+1}|\cdots|dy\le C_*\rho_0(T+1)$ for $T\ge x/(2\omega)$, where $\rho_0:=\log\log J/\log J+(\log(x/2\omega))^{-1/700}$ (note $T\ge x/(2\omega)\ge A_0^{\eps/20}/2\ge U^2\ge J^2$ by \eqref{eq:Abig}). By partial summation, since $F\ge0$,
\[
\sum_{x/\omega<n\le x}\frac{F(n)}n=\frac{S(x)}x-\frac{S(x/\omega)}{x/\omega}+\int_{x/\omega}^x\frac{S(T)}{T^2}\,dT\le C_*\rho_0\,(\log\omega+2).
\]
Dividing by $W\ge\log\omega-1$ with $\log\omega\ge U^4$: $\E F(\bn)\le1.01\,C_*\rho_0$. By \eqref{eq:mass}, $\log(x/2\omega)\ge\frac{\eps}{22}U^4$, and $U\ge\exp(H_-)$, $H_-\ge\eps^{-4}$, $H_-\ge e^{100}$ give $(\frac{\eps}{22}U^4)^{-1/700}\le e^{-3H_-/700}\le\log\log H_-/\log H_-$; also $\log\log J/\log J$ is decreasing for $J\ge16$. Hence $\E F(\bn)\le2.02\,C_*\log\log H_-/\log H_-\le u_-$.

\emph{Step 3: dilation.} For the second bound, $\E[F(c\bn+d)]$: by Lemma~\ref{lem:aff} with $q=c$, $r=d$, $M_X=1$, $\frac1c\E F(c\bn+d)\le\E[F(\bn)\one_{\bn\equiv d\,(c)}]+E_{\rm aff}\le\E F(\bn)+E_{\rm aff}$, so $\E F(c\bn+d)\le cu_-+cE_{\rm aff}\le2cu_-$.

\emph{Step 4: the one-point bound.} Let $h(v):=g(v)\one_{v\equiv d\,(c)}$ and $J:=H_-$. Lemma~\ref{lem:aff} with $q=c$, $r=d$, $X=g$ gives $|\E\,g(c\bn+d)|\le c|\E h(\bn)|+cE_{\rm aff}$. For $1\le j\le J$, Lemma~\ref{lem:aff} with $q=1$, $r=j$, $X=h$ gives $\E h(\bn)=\E h(\bn+j)+O^*(E_{\rm aff})$, so $\E h(\bn)=\E\bigl[\frac1J\sum_{j\le J}h(\bn+j)\bigr]+O^*(E_{\rm aff})$. Expanding the congruence indicator in additive characters modulo $c$,
\[
\frac1J\sum_{j\le J}h(n+j)=\frac1c\sum_{k=0}^{c-1}e\Bigl(\frac{k(n-d)}c\Bigr)\cdot\frac1J\sum_{j\le J}g(n+j)e(kj/c),
\]
so $|\frac1J\sum_{j\le J}h(\bn+j)|\le\frac1c\sum_k|\frac1J\sum_{j\le J}g(\bn+j)e(kj/c)|$, whose expectation is at most $u_-$ by the first bound (with $\alpha=k/c$). Hence $|\E h(\bn)|\le u_-+E_{\rm aff}$ and $|\E\,g(c\bn+d)|\le cu_-+2cE_{\rm aff}\le2cu_-$, as $E_{\rm aff}\le1/(100U^3)\le u_-/2$ (note $u_-\ge3/\log H_-\ge3/U$).
\end{proof}

\section{Two blocks and the multiplicativity step}\label{sec:medium}

For integers $H_-\le H\le H_+$ let $\bY_H:=\bn\bmod P_H$ and
\[
\bX_H:=\Bigl(\bigl(g_1(a_1\bn+r)\bigr)_{1\le r\le a_1KH},\ \bigl(g_2(a_2\bn+r)\bigr)_{1\le r\le a_2KH}\Bigr)\in\{-1,0,1\}^{BH}.
\]
We write $x=(x_{1,r},x_{2,r})$ for a value of $\bX_H$. Then
\begin{equation}\label{eq:HX}
0\le\Hh(\bX_H)\le BH\log3 .
\end{equation}
For $p\in\cP_H$ let
\[
J_p:=\{j\in\Z:\ 1\le a_ij+pb_i\le a_iKH\ \ (i=1,2)\}.
\]
This is an interval of integers: with $t_{i,p}:=\lfloor-pb_i/a_i\rfloor$, the constraint $1\le a_ij+pb_i\le a_iKH$ on the integer $j$ reads $t_{i,p}+1\le j\le t_{i,p}+KH$, so
\[
J_p=\{\max(t_{1,p},t_{2,p})+1,\dots,\min(t_{1,p},t_{2,p})+KH\},\qquad|J_p|=\max\bigl(0,\,KH-|t_{1,p}-t_{2,p}|\bigr).
\]
Since $|pb_i/a_i|\le\vt HR_0\le KH/400$, we have $|t_{1,p}-t_{2,p}|\le p|b_1/a_1-b_2/a_2|+1\le KH/200+1$, hence $0.99KH\le|J_p|\le KH$ (using $KH\ge600$). Define
\begin{equation}\label{eq:F}
F(x,y):=\sum_{p\in\cP_H}F_p(x,y),\qquad F_p(x,y):=\sum_{j\in J_p}\one_{y+j\equiv0\,(p)}\,x_{1,a_1j+pb_1}\,x_{2,a_2j+pb_2}\qquad(y\in\Z/P_H\Z).
\end{equation}
For fixed $y,p$ the $j$ lie in one residue class modulo $p$, so
\begin{equation}\label{eq:Fbound}
|F_p(x,y)|\le\frac{KH}p+1\le\frac{2K}{\vt}+1\le\frac{3K}\vt,\qquad|F(x,y)|\le\frac{3K}\vt\cdot\frac{2\vt H}{\log H}=\frac{6KH}{\log H}.
\end{equation}

\begin{proposition}\label{prop:main}
For every $H_-\le H\le H_+$, $\ \bigl|\E\,F(\bX_H,\bY_H)\bigr|\ge L_0:=\dfrac{\eps KH}{24\log H}$.
\end{proposition}
\begin{proof}
Fix $p\in\cP_H$ and $j\in J_p$. On the event $\bn\equiv-j\ (p)$ write $\bn+j=pm$; then $a_i(\bn+j)+pb_i=p(a_im+b_i)$, so
\[
\one_{\bn+j\equiv0\,(p)}\,g_1(a_1\bn+a_1j+pb_1)\,g_2(a_2\bn+a_2j+pb_2)=\one_{\bn\equiv-j\,(p)}\,g_1\bigl(p\,L_1(m)\bigr)g_2\bigl(p\,L_2(m)\bigr),
\]
where $m:=(\bn+j)/p$. For $g\in\{\lambda,\mu\}$ and $p$ prime, $g(pn)=-g(n)$ unless $p\mid n$, and $|g(pn)+g(n)|\le\one_{p\mid n}$ (with equality to $0$ for $\lambda$). Hence $g_1(pL_1(m))g_2(pL_2(m))=g_1(L_1(m))g_2(L_2(m))+O^*(\one_{p\mid L_1(m)}+\one_{p\mid L_2(m)})$. Lemma~\ref{lem:aff} with $q=p$, $r=-j$ ($|j|\le KH+KH/400\le Q$) gives
\[
\E\bigl[\one_{\bn\equiv-j\,(p)}g_1(L_1(m))g_2(L_2(m))\bigr]=\frac1p\,C_0+O^*(E_{\rm aff}),
\]
and, since $p\nmid a_i$ (by \eqref{eq:Hcond}), the condition $p\mid L_i(m)$ is one residue class modulo $p$ for $m$, so by the last clause of Lemma~\ref{lem:aff} (applied to $m=(\bn+j)/p$ through the same substitution) $\E[\one_{\bn\equiv-j\,(p)}\one_{p\mid L_i(m)}]\le\frac1p(\frac1p+E_{\rm aff})+E_{\rm aff}$. (Explicitly: with $X(n):=\one_{p\mid L_i((n+j)/p)}$ when $p\mid n+j$ and $X:=0$ otherwise, Lemma~\ref{lem:aff} with $q=p$, $r=-j$ gives $\E[X(\bn)\one_{\bn\equiv-j\,(p)}]\le\frac1p\PP(p\mid L_i(\bn))+E_{\rm aff}$, and $p\mid L_i(\bn)$ is the single class $\bn\equiv-b_i\bar a_i\ (p)$.) Altogether the $j$-th term of $\E F_p$ equals $C_0/p+O^*(2/p^2+4E_{\rm aff})$. Summing over $j\in J_p$ and $p\in\cP_H$ ($|J_p|\le KH$ terms each; the lemma is applied three times for each pair $(p,j)$ and twice more for each $p$, for the two classes $p\mid L_i(\bn)$, and the error per term is $2/p^2+(3+2/p)E_{\rm aff}\le2/p^2+4E_{\rm aff}$),
\[
\begin{aligned}
\Bigl|\E F(\bX_H,\bY_H)-C_0\sum_{p\in\cP_H}\frac{|J_p|}p\Bigr|&\le2KH\sum_{p\in\cP_H}\frac1{p^2}+4KH|\cP_H|E_{\rm aff}\\
&\le\frac{16K}{\vt\log H}+\frac{8\vt KH^2}{\log H}\cdot\frac1{100U^3}\le\frac{\eps KH}{100\log H},
\end{aligned}
\]
using $p>\vt H/2$, \eqref{eq:PH}, $H\ge3200/(\eps\vt)$ and $H\le U$. By \eqref{eq:C0}, \eqref{eq:PH} and $|J_p|\ge0.99KH$ the main term has modulus $\ge\frac\eps2\cdot0.99KH\cdot\frac1{4\log H}\ge\frac{\eps KH}{8.1\log H}$, and $\frac1{8.1}-\frac1{100}\ge\frac1{24}$.
\end{proof}

\section{Entropy}\label{sec:entropy}

We use Shannon entropy $\Hh$, conditional entropy, mutual information $\I(\bX;\bY)=\Hh(\bX)+\Hh(\bY)-\Hh(\bX,\bY)\ge0$, relative entropy $\KL(P\|Q)=\sum_zP(z)\log(P(z)/Q(z))$, the standard facts $\Hh(\bX\mid\bY)\le\Hh(\bX)$, $\Hh(\bX,\bY\mid\mathbf Z)\le\Hh(\bX\mid\mathbf Z)+\Hh(\bY\mid\mathbf Z)$, $\Hh(\bX)\le\log N$ for $N$-valued $\bX$, and the continuity estimate \cite{Zha07}: if $P,Q$ are laws on an $N$-point set, $N\ge2$, with $\TV(P,Q)=\eta\le1/2$ then $|\Hh(P)-\Hh(Q)|\le\eta\log(N-1)+h_2(\eta)$, where $h_2(t):=-t\log t-(1-t)\log(1-t)$, with $0\log0:=0$. For completeness, choose a maximal coupling $(V,W)$ of $P,Q$ with $\PP(V\ne W)=\eta$, and put $E:=\one_{V\ne W}$. Then
\[
\Hh(V)-\Hh(W)\le\Hh(V\mid W)\le\Hh(E\mid W)+\Hh(V\mid E,W)\le h_2(\eta)+\eta\log(N-1).
\]
Interchanging $V,W$ proves the estimate. The case $N=1$ is immediate.

\begin{lemma}[Near-uniformity of $\bY_H$]\label{lem:unif}
For $H_-\le H\le H_+$, $\rho_H:=\log P_H-\Hh(\bY_H)$ satisfies $0\le\rho_H\le e^{-2U}$.
\end{lemma}
\begin{proof}
For each $r$, $\PP(\bY_H=r)=\frac1W\sum_{n\equiv r\,(P_H)}\frac1n=\frac1{P_H}(1+\eta_r)$ with $|\eta_r|\le\eta:=4P_H\omega/(x\log\omega)$ (compare the sum over the progression with $\frac1{P_H}\int_{x/\omega}^xdt/t$, whose endpoint terms are $\le2\omega/x$ each, and use \eqref{eq:mass} for $W$) and $\sum_r\eta_r=0$. By \eqref{eq:PH}, \eqref{eq:mass} and \eqref{eq:Abig}, $\eta\le4e^{1.02\vt U}e^{-\eps U^4/20}\le e^{-U}$. Then $\rho_H=\sum_r\frac{1+\eta_r}{P_H}\log(1+\eta_r)\le\sum_r\frac{\eta_r+\eta_r^2}{P_H}\le\eta^2$, using $(1+u)\log(1+u)\le u+u^2$ for $|u|\le\frac12$; and $\rho_H\ge0$ as $\Hh\le\log P_H$.
\end{proof}

\begin{lemma}[Relative subadditivity]\label{lem:subadd}
Let $H,k$ be integers with $H_-\le H\le kH\le H_+$. Then
\begin{equation}\label{eq:3.12}
\frac{\Hh(\bX_{kH})}{kH}\le\frac{\Hh(\bX_H)}H-\frac{\I(\bX_H;\bY_H)}H+\frac{1.02\,\vt}k+\frac3H .
\end{equation}
\end{lemma}
\begin{proof}
For $0\le i<k$ put $s_i:=iKH$ and $\bX^{(i)}:=\bigl((g_1(a_1\bn+a_1s_i+r))_{r\le a_1KH},(g_2(a_2\bn+a_2s_i+r))_{r\le a_2KH}\bigr)$; then $\bX_{kH}$ is (a rearrangement of) the concatenation of $\bX^{(0)},\dots,\bX^{(k-1)}$. Let $\Psi(m):=(\Phi(m),m\bmod P_H)$ where $\Phi(m)$ is the pair of blocks at $m$, so that $\Psi(\bn)=(\bX_H,\bY_H)$ and $\Psi(\bn+s_i)=(\bX^{(i)},\bY_H+s_i)$. Since $y\mapsto y+s_i$ is a bijection of $\Z/P_H\Z$, both $\Hh(\bX^{(i)}\mid\bY_H)=\Hh(\bX^{(i)}\mid\bY_H+s_i)$ and $\Hh(\bY_H+s_i)=\Hh(\bY_H)$ hold exactly. By Lemma~\ref{lem:aff}, $\TV(\bn,\bn+s_i)\le E_{\rm aff}$ ($s_i\le KH_+\le Q$), hence by data processing $\TV(\Psi(\bn),\Psi(\bn+s_i))\le E_{\rm aff}$. The marginal entropies therefore cancel when the conditional entropies are written as joint entropy minus marginal entropy. The joint alphabet has size at most $3^{BH}P_H$, with logarithm at most $BH\log3+1.02\vt H\le1.11U$, so one application of the continuity estimate gives
\[
\begin{aligned}
\bigl|\Hh(\bX^{(i)}\mid\bY_H)-\Hh(\bX_H\mid\bY_H)\bigr|
&\le E_{\rm aff}\log(3^{BH}P_H-1)+h_2(E_{\rm aff})\\
&\le1.11U\,E_{\rm aff}+h_2(E_{\rm aff})\le\frac3U\le\frac3k .
\end{aligned}
\]
Relative subadditivity gives $\Hh(\bX_{kH}\mid\bY_H)\le\sum_i\Hh(\bX^{(i)}\mid\bY_H)\le k\Hh(\bX_H\mid\bY_H)+3$, so
\[
\Hh(\bX_{kH})=\Hh(\bX_{kH}\mid\bY_H)+\Hh(\bY_H)-\Hh(\bY_H\mid\bX_{kH})\le k\Hh(\bX_H)-k\,\I(\bX_H;\bY_H)+\log P_H+3 .
\]
Divide by $kH$ and use \eqref{eq:PH}.
\end{proof}

\begin{proposition}[Entropy decrement]\label{prop:decrement}
There is an integer $H\in[H_-,H_+]$ such that $\I(\bX_H;\bY_H)\le\delta H/\log H$.
\end{proposition}
\begin{proof}
Suppose not. Let $H_1:=H_-$, $k_j:=\lceil4.08\,\vt\log H_j/\delta\rceil\ge2$, $H_{j+1}:=k_jH_j$; then $1.02\vt/k_j\le\delta/(4\log H_j)$. While $H_{j+1}\le H_+$, Lemma~\ref{lem:subadd} and the assumption give
\[
\frac{\Hh(\bX_{H_{j+1}})}{H_{j+1}}\le\frac{\Hh(\bX_{H_j})}{H_j}-\frac\delta{\log H_j}+\frac\delta{4\log H_j}+\frac3{H_j}\le\frac{\Hh(\bX_{H_j})}{H_j}-\frac\delta{2\log H_j},
\]
as $3/H_j\le\delta/(4\log H_j)$ by \eqref{eq:Hcond}. Telescoping from \eqref{eq:HX} with $\Hh\ge0$: if $H_J\le H_+$ then
\begin{equation}\label{eq:tele}
\sum_{j<J}\frac{\delta}{2\log H_j}\le B\log3 .
\end{equation}
Put $x_j:=\log H_j$. Then $x_{j+1}\le x_j+\log(8.16\,\vt x_j/\delta)\le x_j+\log x_j+L$ with $L=\log(9\vt/\delta)$. As $u\mapsto1/(u(\log u+L))$ is decreasing on $u\ge1$,
\[
\frac1{x_j}\ge\frac{x_{j+1}-x_j}{x_j(\log x_j+L)}\ge\int_{x_j}^{x_{j+1}}\frac{du}{u(\log u+L)},\qquad\text{so}\qquad\sum_{j<J}\frac1{x_j}\ge\log\frac{\log x_J+L}{\log x_1+L}.
\]
Since $k_j\ge2$, $x_J\to\infty$ and the right side tends to infinity with $J$; let $J$ be the least index with $\sum_{j<J}1/x_j>S=2B\log3/\delta$. Then $\sum_{j<J-1}1/x_j\le S$, whence $\log x_{J-1}\le(\log x_1+L)e^S-L$ and $x_J\le x_{J-1}+\log x_{J-1}+L\le2\exp\bigl((\log x_1+L)e^S\bigr)$, i.e.\ $H_J\le H_+$ by \eqref{eq:Hplus} ($x_1=\log H_-$). But then \eqref{eq:tele} holds while $\sum_{j<J}\delta/(2x_j)>B\log3$, a contradiction.
\end{proof}

\begin{lemma}[Entropy to expectation]\label{lem:ent}
Let $(\bX,\bY)$ be finite-valued random variables, $\bY$ taking values in a finite set $\mathcal Y=\prod_p\mathcal Y_p$, let $U_{\mathcal Y}$ be the uniform law on $\mathcal Y$, $Q:=\mathrm{Law}(\bX)\otimes U_{\mathcal Y}$, $P:=\mathrm{Law}(\bX,\bY)$. Suppose $F(x,y)=\sum_pF_p(x,y_p)$ with $F_p$ real and $|F_p|\le B_p$, and put $V:=\sum_pB_p^2$. Then
\[
\bigl|\E_PF-\E_QF\bigr|\le\sqrt{2V\,\KL(P\|Q)},\qquad \KL(P\|Q)=\I(\bX;\bY)+\bigl(\log|\mathcal Y|-\Hh(\bY)\bigr).
\]
\end{lemma}
\begin{proof}
The formula for $\KL$: $\KL(P\|Q)=\sum P(x,y)\log\frac{P(x,y)}{P_{\bX}(x)/|\mathcal Y|}=-\Hh(\bX,\bY)+\Hh(\bX)+\log|\mathcal Y|$. Let $Z(x,y):=F(x,y)-\E_{U_{\mathcal Y}}F(x,\cdot)$. For fixed $x$, under $U_{\mathcal Y}$ the summands $F_p(x,y_p)-\E F_p(x,\cdot)$ are independent, mean zero, with range in an interval of length $2B_p$, so Hoeffding's lemma gives $\E_{U_{\mathcal Y}}e^{tZ(x,\cdot)}\le e^{t^2V/2}$ for real $t$; averaging over $x$, $\log\E_Qe^{tZ}\le t^2V/2$. The Donsker--Varadhan inequality $\E_PZ\le\frac1t(\KL(P\|Q)+\log\E_Qe^{tZ})$ ($t>0$) gives $\E_PZ\le\KL/t+tV/2$; if $V=0$ then $F\equiv0$ and there is nothing to prove; if $\KL>0$ take $t=\sqrt{2\KL/V}$; if $\KL=0$ then $P=Q$ and there is nothing to prove. Apply the same to $-Z$. Finally $\E_PZ=\E_PF-\E_QF$ since $\E_P[\E_{U_{\mathcal Y}}F(\bX,\cdot)]=\E_QF$.
\end{proof}

\section{Decoupling, the circle method, and the contradiction}\label{sec:circle}

Fix $H$ as in Proposition~\ref{prop:decrement}. Apply Lemma~\ref{lem:ent} with $\mathcal Y=\Z/P_H\Z=\prod_{p\in\cP_H}\Z/p\Z$, $F$ as in \eqref{eq:F} ($F_p$ depends on $y$ only through $y\bmod p$), $B_p=3K/\vt$. Then $V\le|\cP_H|(3K/\vt)^2\le18K^2H/(\vt\log H)$ by \eqref{eq:PH}, and $\KL\le\delta H/\log H+\rho_H\le2\delta H/\log H$ by Lemma~\ref{lem:unif}. With \eqref{eq:delta},
\[
\bigl|\E F(\bX_H,\bY_H)-\E_QF\bigr|\le\sqrt{\frac{72\,\delta K^2H^2}{\vt\log^2H}}=\frac{\sqrt{72}}{1440}\cdot\frac{\eps KH}{\log H}\le\frac{L_0}{7}
\]
(exactly, $(\sqrt{72}/1440)^2/(1/168)^2=49/50$). Since $\E_{U_{\mathcal Y}}F_p(x,\cdot)=\frac1p\sum_{j\in J_p}x_{1,a_1j+pb_1}x_{2,a_2j+pb_2}$, Proposition~\ref{prop:main} gives
\begin{equation}\label{eq:3.17}
|\E\,T(\bX_H)|\ge\frac67L_0,\qquad T(x):=\sum_{p\in\cP_H}\frac1p\sum_{j\in J_p}x_{1,a_1j+pb_1}\,x_{2,a_2j+pb_2}.
\end{equation}

\begin{lemma}[Circle method on zero-padded blocks]\label{lem:circle}
Let $M:=2KH$, $D_i:=a_iM$, $N:=aM$, $S_H(\alpha):=\sum_{p\in\cP_H}e(\alpha p)/p$, $\tau:=\eps/(480\sqrt a)$. Extend each block by $x_{i,r}:=0$ for $a_iKH<r\le D_i$ and periodically modulo $D_i$, and put $G_i(\xi):=\frac1{D_i}\sum_{r=1}^{a_iKH}x_{i,r}e(-r\xi/D_i)$ for $\xi\in\Z/D_i\Z$. Let
\[
\Xi:=\Bigl\{\xi\in\Z/D_1\Z:\ \Bigl|S_H\Bigl(-\frac{\Delta\xi}N+\frac{b_2\eta}{a_2}\Bigr)\Bigr|\ge\frac\tau{\log H}\text{ for some }\eta\in\Z/a_2\Z\Bigr\}.
\]
Here and in \eqref{eq:fourier} choose the representatives $0\le\xi<D_1$ and $0\le\eta<a_2$. Replacing $\xi$ by $\xi+D_1$ and reindexing $\eta$ by $\eta+a_1$ leaves $-\xi+M\eta$ unchanged modulo $D_2$ and changes the phase by an integer. Thus both the definition of $\Xi$ and the inner sum in \eqref{eq:fourier} are independent of the choice of representative. Then for every $x$,
\[
|T(x)|\le\frac{\eps KH}{480\log H}+\frac{4a_2KH}{\log H}\sum_{\xi\in\Xi}|G_1(\xi)| .
\]
\end{lemma}
\begin{proof}
\emph{Exactness of padding.} Let $\mathcal R:=\{-\lfloor KH/2\rfloor,\dots,M-\lfloor KH/2\rfloor-1\}$, a complete residue system modulo $M$ containing $J_p$ (as $J_p\subset[-KH/400,\,KH+KH/400]$). For $j\in\mathcal R$ the index $a_ij+pb_i$ lies in $[-0.51a_iKH,\,1.51a_iKH]$ (using $|pb_i|\le a_iKH/400$); reduced modulo $D_i=2a_iKH$, a negative index lands in $(1.49a_iKH,D_i]$ and an index $>a_iKH$ stays in $(a_iKH,1.51a_iKH]$, both in the zero-padded half. The index zero also has value zero, since $x_{i,0}=x_{i,D_i}=0$. Hence $x_{1,a_1j+pb_1}x_{2,a_2j+pb_2}$ (indices read modulo $D_i$) vanishes for $j\in\mathcal R\setminus J_p$, and
$T(x)=\sum_p\frac1p\sum_{j\in\mathcal R}x_{1,a_1j+pb_1}x_{2,a_2j+pb_2}$ exactly.

\emph{Fourier expansion.} With $x_{i,r}=\sum_{\xi_i\in\Z/D_i\Z}G_i(\xi_i)e(r\xi_i/D_i)$ and $r=a_ij+pb_i$, $e(r\xi_i/D_i)=e(j\xi_i/M)e(pb_i\xi_i/(a_iM))$. Summing over $j\in\mathcal R$: $\sum_{j\in\mathcal R}e(j(\xi_1+\xi_2)/M)=M\one_{\xi_1+\xi_2\equiv0\,(M)}$. Write $\xi_2=-\xi_1+M\eta$, $\eta\in\Z/a_2\Z$; the phase becomes $p\bigl(\frac{b_1\xi_1}{a_1M}-\frac{b_2\xi_1}{a_2M}+\frac{b_2\eta}{a_2}\bigr)=p\bigl(-\frac{\Delta\xi_1}{N}+\frac{b_2\eta}{a_2}\bigr)$. Hence, exactly,
\begin{equation}\label{eq:fourier}
T(x)=M\sum_{\xi\in\Z/D_1\Z}\ \sum_{\eta\in\Z/a_2\Z}G_1(\xi)\,G_2(-\xi+M\eta)\,S_H\Bigl(-\frac{\Delta\xi}N+\frac{b_2\eta}{a_2}\Bigr).
\end{equation}
\emph{Bounds.} Plancherel gives $\sum_{\xi}|G_i(\xi)|^2=\frac1{D_i}\sum_{r\le a_iKH}|x_{i,r}|^2\le\frac12$ and $|G_i|\le\frac12$. As $(\xi,\eta)$ ranges over $\Z/D_1\Z\times\Z/a_2\Z$, each value of $-\xi+M\eta\in\Z/D_2\Z$ is taken $a_1$ times, so by Cauchy--Schwarz $\sum_{\xi,\eta}|G_1(\xi)||G_2(-\xi+M\eta)|\le\sqrt{a_2/2}\sqrt{a_1/2}=\sqrt a/2$. For $\xi\notin\Xi$ use $|S_H|<\tau/\log H$: the total is $\le M\cdot\frac{\sqrt a}2\cdot\frac\tau{\log H}=\frac{\eps KH}{480\log H}$, as $\tau\sqrt a=\eps/480$. For $\xi\in\Xi$ use $|G_2|\le\frac12$ and $|S_H|\le\sum_{p\in\cP_H}1/p\le4/\log H$: total $\le M\cdot a_2\cdot\frac12\cdot\frac4{\log H}\sum_{\xi\in\Xi}|G_1(\xi)|=\frac{4a_2KH}{\log H}\sum_{\xi\in\Xi}|G_1(\xi)|$.
\end{proof}

Once $H$ is chosen, $\Xi$ depends only on $H$, the affine forms, the core tolerance $\eps$, and the fixed prime band; it does not depend on the realization of $\bX_H$. We apply Lemma~\ref{lem:mrt} separately at each fixed frequency $-\xi/D_1$ and sum the resulting expectation bounds. Apply Lemma~\ref{lem:circle} with $x=\bX_H$ and take expectations in \eqref{eq:3.17}. Since $\eps KH/(480\log H)=L_0/20$,
\begin{equation}\label{eq:Xi}
\sum_{\xi\in\Xi}\E|G_1(\xi)|\ge\Bigl(\frac67-\frac1{20}\Bigr)\frac{L_0\log H}{4a_2KH}\ge\frac{4}{5}\cdot\frac{\eps}{96a_2}\ge\frac{\eps}{120a_2}.
\end{equation}
(If $\Xi=\emptyset$ this already contradicts \eqref{eq:3.17}.) On the other hand $G_1(\xi)=\frac12\cdot\frac1{a_1KH}\sum_{r\le a_1KH}g_1(a_1\bn+r)e(-r\xi/D_1)$, so by Lemma~\ref{lem:mrt} (with $J=a_1KH\in[H_-,U]$, $c=a_1$, $d=0$), $\E|G_1(\xi)|\le a_1u_-$. Therefore
\begin{equation}\label{eq:mu}
u_-\ge\frac{\eps}{120\,a\,|\Xi|}.
\end{equation}

\begin{lemma}[Restriction]\label{lem:restr}
Let $d_H:=\gcd(N,\Delta,a_1b_2M)$. Then $|\Xi|\le C_5\,d_H\,a^4K^2\vt^{-2}\eps^{-4}\le C_5\,|\Delta|\,a^4K^2\vt^{-2}\eps^{-4}$.
\end{lemma}
\begin{proof}
Apply \cite[Proposition 4.2]{GT06} with $k=1$, $F(n)=n$ (singular series $1$), $N=2aKH\ge N_{\rm GT}$, $R:=\lfloor N^{1/10}\rfloor$, exponent $4$, and $a_n:=1/(n\beta_R(n))$ for $n\in\cP_H$, $a_n:=0$ otherwise. Every $p\in\cP_H$ exceeds $R$ by \eqref{eq:Hcond}, so $p$ lies in the sifted set of \cite[Prop.\ 3.1(i)]{GT06} and $\beta_R(p)\ge c_{\rm GT}\log R\ge(c_{\rm GT}/20)\log H$ (as $R\ge N^{1/10}/2\ge H^{1/20}$). The proposition gives
\[
\sum_{k\in\Z/N\Z}\Bigl|S_H\Bigl(\frac kN\Bigr)\Bigr|^4\le C_{\rm GT}^4N^2\Bigl(\sum_{p\in\cP_H}\frac1{p^2\beta_R(p)}\Bigr)^2\le4C_{\rm GT}^4a^2K^2H^2\Bigl(\frac{160}{c_{\rm GT}\vt H\log^2H}\Bigr)^2,
\]
using $\frac{20}{c_{\rm GT}\log H}\cdot\frac{2\vt H}{\log H}\cdot\frac4{\vt^2H^2}=\frac{160}{c_{\rm GT}\vt H\log^2H}$; the right side equals $\frac{102400\,C_{\rm GT}^4a^2K^2}{c_{\rm GT}^2\vt^2\log^4H}$. By Markov, the set $\cA_H$ of $k\in\Z/N\Z$ with $|S_H(k/N)|\ge\tau/\log H$ has at most $102400C_{\rm GT}^4c_{\rm GT}^{-2}a^2K^2\vt^{-2}(480\sqrt a/\eps)^4=C_5a^4K^2\vt^{-2}\eps^{-4}$ elements.

Let $\mathcal E:=\{(\xi_1,\xi_2)\in\Z/D_1\Z\times\Z/D_2\Z:\ \xi_1+\xi_2\equiv0\ (M)\}$, a subgroup of order $N$ (each $\xi_1$ has the $a_2$ partners $\xi_2=-\xi_1+M\eta$), and let $\Phi(\xi_1,\xi_2):=a_2b_1\xi_1+a_1b_2\xi_2\bmod N$; this is a well-defined homomorphism $\mathcal E\to\Z/N\Z$, since $a_2b_1D_1=b_1N$ and $a_1b_2D_2=b_2N$. For $\xi_2=-\xi+M\eta$ one has $\Phi(\xi,\xi_2)=-\Delta\xi+a_1b_2M\eta$, and the phase in the definition of $\Xi$ is exactly $\Phi(\xi,\xi_2)/N$ (as $b_2\eta/a_2=a_1b_2M\eta/N$); so $\Xi$ is the projection to the first coordinate of $\Phi^{-1}(\cA_H)$. The group $\mathcal E$ is generated by $(1,-1)$ and $(0,M)$ (given $(\xi_1,\xi_2)\in\mathcal E$, $(\xi_1,\xi_2)-\xi_1(1,-1)=(0,\xi_1+\xi_2)$ is a multiple of $(0,M)$), whose images are $-\Delta$ and $a_1b_2M$; hence $\Phi(\mathcal E)$ is the subgroup of $\Z/N\Z$ generated by $\gcd(\Delta,a_1b_2M)$, of order $N/d_H$, and every fibre of $\Phi$ has exactly $d_H$ elements. Therefore $|\Xi|\le|\Phi^{-1}(\cA_H)|\le d_H|\cA_H|$; and $d_H\le|\Delta|$ since $d_H\mid\Delta\ne0$.
\end{proof}

Explicitly, $\gcd(N,a_1b_2M)=a_1M\gcd(a_2,b_2)$, so $d_H=\gcd(\Delta,\,a_1M\gcd(a_2,b_2))$; thus $d_H=1$ whenever $\gcd(\Delta,M)=\gcd(\Delta,a_1\gcd(a_2,b_2))=1$ (for instance when $|\Delta|=1$), while $d_H\ge2$ whenever $\Delta$ is even, as $M=2KH$. We carry $|\Delta|$ rather than $d_H$ into $B_0$ because $B_0$ is fixed before $H$ is selected in Proposition~\ref{prop:decrement}, and $d_H=|\Delta|$ for every $H\in[H_-,H_+]$ divisible by $|\Delta|$.

\begin{proof}[Proof of Theorem~\ref{thm:A}]
By \eqref{eq:mu}, Lemma~\ref{lem:restr} and the definition of $u_-$,
\[
3C_*\frac{\log\log H_-}{\log H_-}\ge\frac{\eps^5\vt^2}{120\,C_5\,a^5|\Delta|K^2},\qquad\text{i.e.}\qquad\frac{\log H_-}{\log\log H_-}\le\frac{360\,C_*C_5\,a^5|\Delta|K^2}{\vt^{2}\eps^{5}}\le B_0,
\]
contradicting the strict inequality \eqref{eq:Hminusprop}. Hence \eqref{eq:contra} is impossible: \eqref{eq:main} holds under \eqref{eq:standing} for $\omega\ge A_0$, and by \S\ref{sec:red} for all $\omega\ge A_0(\eps/2)^2=A$, with $\logk4A\le C_1B\eps^{-2}$ by \eqref{eq:Asize}. The rate form was derived after \eqref{eq:Asize}.
\end{proof}

\newpage
\section{Sign pairs and the centred theorem}\label{sec:centred}

\subsection{Sign pairs}
\begin{corollary}\label{cor:signs}
Under the hypotheses of Theorem~\ref{thm:A}, for each $(s,t)\in\{-1,1\}^2$,
\begin{equation}\label{eq:signs}
\Bigl|\sum_{x/\omega<n\le x}\frac{\one_{\lambda(L_1(n))=s,\ \lambda(L_2(n))=t}-\frac14}{n}\Bigr|\le\eps\log\omega,
\end{equation}
and the rate form \eqref{eq:rate} holds for the same sums. Thus each of the four sign pairs has logarithmic frequency $\frac14+O((\fh/\logk4\omega)^{1/2})$ in every window $(x/\omega,x]$ with $\logk4\omega\ge C\fh$, uniformly over admissible pairs of height at most $\fh$.
\end{corollary}
\begin{proof}
Put $f(n):=\one_{\lambda(L_1(n))=s,\lambda(L_2(n))=t}-\frac14$, so $|f|\le1$. By Lemma~\ref{lem:window}, exactly as in \S\ref{sec:red}, it suffices to prove $|\sum_{x/\omega<n\le x}f(n)/n|\le\eps\log\omega$ with the halved $\eps$, under \eqref{eq:standing} and for $\omega\ge A_0$, where every $L_i(n)\ge1$. There $f(n)=\frac14\bigl(s\lambda(L_1(n))+t\lambda(L_2(n))+st\,\lambda(L_1(n))\lambda(L_2(n))\bigr)$. Let $\bn$ be as in \S\ref{sec:red}. The proof of Theorem~\ref{thm:A} shows that \eqref{eq:contra} is impossible, so $|\E\,\lambda(L_1(\bn))\lambda(L_2(\bn))|\le\eps\log\omega/W\le1.01\eps$; and the third bound of Lemma~\ref{lem:mrt} with $c=a_i$, $d=b_i$ gives $|\E\,\lambda(L_i(\bn))|\le2a_iu_-\le\eps$, because $u_-<3C_*/B_0$ by \eqref{eq:Hminusprop} and $6a_iC_*\le\eps B_0$ by \eqref{eq:Hminus}. Hence $|\E f(\bn)|\le\frac14(2\eps+1.01\eps)\le0.76\eps$, and $|\sum_{x/\omega<n\le x}f(n)/n|=W|\E f(\bn)|\le0.76\eps(\log\omega+1)\le\eps\log\omega$.
\end{proof}

\subsection{Translation}
\begin{lemma}[Translation]\label{lem:trans}
Let $0<\eps\le1/10$, $A\ge1$, $m\in\Z$, and $f:\Z\to\mathbb C$ with $|f|\le1$. Suppose that
\[
\Bigl|\sum_{X/\Omega<n\le X}\frac{f(n)}n\Bigr|
\le\frac\eps2\log\Omega
\qquad\text{for all real }X\ge\Omega\ge A.
\]
Then for all
\[
x\ge\omega\ge
\max\bigl\{A^2,\ 8[4(|m|+1)]^{2/\eps-1}\bigr\},
\]
one has
\[
\Bigl|\sum_{x/\omega<n\le x}\frac{f(n+m)}n\Bigr|
\le\eps\log\omega.
\]
\end{lemma}
\begin{proof}
Put $\alpha:=\eps/2$, $M_0:=|m|$, $B_1:=2M_0+2$,
$u_0:=x/\omega$, $j:=\lfloor u_0\rfloor\ge1$, and
$K_1:=\max(j,B_1)$. The threshold implies $\omega\ge4B_1^2$,
so $K_1\le x$. Discarding $j<n\le K_1$ costs at most
$H_{K_1}-H_j\le\log(K_1/j)$. Replacing $1/n$ by $1/(n+m)$
for $n>K_1$ costs at most $\log2$, by the telescoping identities
\[
\sum_{n>K_1}\left(\frac1n-\frac1{n+M_0}\right)
=H_{K_1+M_0}-H_{K_1},\qquad
\sum_{n>K_1}\left(\frac1{n-M_0}-\frac1n\right)
=H_{K_1}-H_{K_1-M_0}.
\]
Both are at most $\log2$, since $K_1\ge2M_0+2$.
The remaining translated sum has upper endpoint $X:=x+m$,
lower endpoint $K_1+m\ge1$, and ratio
$\Omega:=(x+m)/(K_1+m)$. Because $|m|\le K_1/2$ and
$K_1\le x$, the two sign cases give
\[
\frac{x}{2K_1}\le\Omega\le\frac{2x}{K_1},\qquad X\ge\Omega.
\]
Furthermore $x/K_1\ge\omega/B_1$, so
$\Omega\ge\omega/(2B_1)\ge\sqrt\omega\ge A$.
The hypothesis and these inequalities give
\begin{align*}
\Bigl|\sum_{x/\omega<n\le x}\frac{f(n+m)}n\Bigr|
&\le\log(K_1/j)+\log2+\alpha\log(2x/K_1)\\
&=\alpha\log\omega+(1-\alpha)\log(K_1/j)
 +\alpha\log(u_0/j)+(1+\alpha)\log2\\
&\le\alpha\log\omega+(1-\alpha)\log B_1
 +(1+2\alpha)\log2,
\end{align*}
since $K_1/j\le B_1$ and $u_0/j<2$. Finally
$\omega\ge8(2B_1)^{2/\eps-1}$ is exactly sufficient for
\[
\alpha\log\omega\ge(1-\alpha)\log B_1+(1+2\alpha)\log2.
\]
This proves the claim, including nonintegral $x/\omega$.
\end{proof}

\begin{remark}[Sharpness of the black-box translation exponent]\label{rem:trans}
For fixed $\eps$, the exponent $2/\eps-1$ cannot be reduced in
a translation theorem with source tolerance $\eps/2$, even if the
prefactor is allowed to depend on $A$ and $\eps$.
Indeed, write $\alpha:=\eps/2$ and choose $0\le\beta<\alpha$.
For an integer $M_0\ge2$, let
\[
f_{M_0}(v):=\beta+(1-\beta)\one_{M_0<v\le2M_0}.
\]
Every source window satisfies
\[
\sum_{X/\Omega<n\le X}\frac{f_{M_0}(n)}n
\le\beta\log\Omega+\beta+(1-\beta)\log2.
\]
Thus the hypothesis of Lemma~\ref{lem:trans} holds, uniformly in
$M_0$, with
\[
A:=\exp\left(\frac{\beta+(1-\beta)\log2}{\alpha-\beta}\right).
\]
For $m=M_0$ and $x=\omega\ge M_0$, the translated sum is exactly
\[
\beta(H_{\lfloor\omega\rfloor}-1)+(1-\beta)(H_{M_0}-1).
\]
If $\omega=C M_0^r$ with $r>1$ and fixed $C>0$, its ratio to
$\log\omega$ tends to $\beta+(1-\beta)/r$. Letting
$\beta\uparrow\alpha$ shows that a necessary exponent is
\[
r\ge\frac{1-\alpha}{\eps-\alpha}=\frac2\eps-1.
\]
An asserted exponent $r\le1$ fails as well, by testing at a larger
exponent strictly between $1$ and $2/\eps-1$.
For inheritance of \emph{vanishing} normalized cancellation, in the
worst case over bounded families having the corresponding source
cancellation uniformly, $\log(2+|m|)=o(\log\omega)$ is the exact
translation-size condition. This is not a necessary condition for each
individual $f$ or for one fixed positive tolerance. The special case
$\beta=0$ recovers the band indicator obstruction, with
$A=e^{2\log2/\eps}$ and translated sum $H_{M_0}-1$.
No arithmetic property of $\lambda$ is used in this obstruction.
\end{remark}
\begin{remark}[Source tolerance versus translation]\label{rem:tradeoff}
The exponent $2/\eps-1$ is tied to the source tolerance $\eps/2$; other allocations of the error budget give other exponents. If the hypothesis of Lemma~\ref{lem:trans} holds with $\alpha\log\Omega$ in place of $\frac\eps2\log\Omega$, $0<\alpha<\eps$, the same proof gives the bound $\alpha\log\omega+(1-\alpha)\log B_1+(1+2\alpha)\log2$ for $\omega\ge\max\{A^2,4B_1^2\}$, hence the conclusion at tolerance $\eps$ for
\[
\omega\ge\max\bigl\{A^2,\ 2^{3\alpha/(\eps-\alpha)}[4(|m|+1)]^{(1-\alpha)/(\eps-\alpha)}\bigr\},
\]
(the displayed threshold implies $\omega\ge4B_1^2$, because $(1-\alpha)/(\eps-\alpha)\ge1/\eps\ge10$), and the family of Remark~\ref{rem:trans} with $\beta\uparrow\alpha$ shows that the exponent $(1-\alpha)/(\eps-\alpha)$ is sharp for that source tolerance. In Theorem~\ref{thm:C} one may therefore apply Theorem~\ref{thm:A} at tolerance $\theta\eps$, $0<\theta<1$, at the price of the tower coefficient: since $B'\le\frac32B_s$, the threshold \eqref{eq:threshC} may be replaced by
\[
\omega\ge\max\Bigl\{\expk4\Bigl(\Bigl(\frac{3C_1}{2\theta^2}+1\Bigr)B_s\eps^{-2}\Bigr),\ 2^{3\theta/(1-\theta)}[4(|m|+1)]^{(1-\theta\eps)/((1-\theta)\eps)}\Bigr\}.
\]
As $\theta\downarrow0$, the leading coefficient $1/(1-\theta)$ of $1/\eps$ in the exponent approaches $1$, which the band example shows cannot be beaten, while the tower coefficient grows like $\theta^{-2}$; the choice $\theta=\eps$ gives exponent $1/\eps+1$ but a tower cost of order $\eps^{-4}$. These are trade-offs, not improvements of \eqref{eq:rateC}.
\end{remark}

\subsection{Proof of Theorem~\ref{thm:C}}
\emph{The bound $B_s\le8\fh_s$.} If $d\le2$ then $K_s=4$ and
$B_s\le8\max(a_1,a_2)$. If $d>2$ then $K_s\le2d+1\le\frac52d$
and $B_s\le\frac52|\Delta|(a_1+a_2)/(a_1a_2)\le5|\Delta|$.

\emph{Centring.} Put $b_i':=b_i-a_im$ and
$L_i'(n):=a_in+b_i'=L_i(n-m)$. The pair $(L_1',L_2')$ is admissible with the same determinant $\Delta$,
and $|b_i'|/a_i=|\beta_i-m|\le|\beta_i-\frac{\beta_1+\beta_2}2|+\frac12=\frac d2+\frac12$. Thus the centred block parameter satisfies
$K'\le\lceil4\max(1,d/2)\rceil+2=K_s+2\le\frac32K_s$ (as $K_s\ge4$) and $B'\le\frac32B_s$.
Theorem~\ref{thm:A}, or Corollary~\ref{cor:signs}, applied at tolerance
$\eps/2$ supplies the hypothesis of Lemma~\ref{lem:trans} with
\[
A:=\expk4(4C_1B'\eps^{-2})\le\expk4(6C_1B_s\eps^{-2}).
\]
Since $f(n+m)$ is the desired correlation, or sign indicator minus
$1/4$, the translation lemma proves the assertion for
$\omega\ge\max\{A^2,8[4(|m|+1)]^{2/\eps-1}\}$.
Also
\[
A^2\le\expk4(6C_1B_s\eps^{-2}+1)
\le\expk4((6C_1+1)B_s\eps^{-2}),
\]
because $(\expk4u)^2\le\expk4(u+1)$ for $u\ge0$.

\emph{The rate form.} Put
\[
\eps_0:=\max\left\{
\left(\frac{(6C_1+1)B_s}{\logk4\omega}\right)^{1/2},
\frac{4\log(4|m|+4)}{\log\omega}\right\}.
\]
The rate domain $\logk4\omega\ge1$ ensures $\omega\ge4$. If $\eps_0\le1/10$, the translation threshold in \eqref{eq:threshC} holds at $\eps=\eps_0$, since
\[
8[4(|m|+1)]^{2/\eps_0-1}
=\frac{2[4(|m|+1)]^{2/\eps_0}}{|m|+1}
\le2\omega^{1/2}\le\omega.
\]
In this case the theorem bounds the sum by
$\eps_0\log\omega$, which is at most
\[
\left(\frac{(6C_1+1)B_s}{\logk4\omega}\right)^{1/2}\log\omega
+4\log(4|m|+4).
\]
If $\eps_0>1/10$, use the trivial bound
$\log\omega+1\le20\eps_0\log\omega$.
Now $B_s\le8\fh_s$ and
$\log(4|m|+4)\le2\log(2+|m|)$ give \eqref{eq:rateC}. \qed

\section{Remarks}\label{sec:remarks}

\begin{enumerate}[leftmargin=1.6em]
\item \emph{Where the tower comes from.} $\logk3H_+\asymp S=2B\log3/\delta$: the residue entropy per rung is $\asymp\vt H$ while the usable mutual information is $\asymp\delta H/\log H$, so each rung needs $k\gtrsim(\vt/\delta)\log H$, and the entropy accumulated along the ladder is $\delta\int du/(u(\log u+L))\asymp\delta\logk3H$, which must exhaust the normalized entropy budget $B\log3$. The fourth exponential is $\log A_0\gg U=BH_+$, needed for $\bn$ to be equidistributed modulo $P_H$ with $\log P_H\asymp\vt H$ and for the affine errors. We do not see how to remove either exponential from this argument; the linear dependence on $B$ enters only through the entropy budget. The $\eps^{-2}$ is the Hoeffding scaling of Lemma~\ref{lem:ent}. Replacing Hoeffding's lemma there by a Bernstein bound, with the variance of $F_p(x,\cdot)$ under the uniform law in place of its range, does not improve the exponent in general: for fixed $p$ the map $j\mapsto a_1j+pb_1$ is injective, so one may take the second block identically $1$ and prescribe the first so that $x_{1,a_1j+pb_1}x_{2,a_2j+pb_2}=s_{j\bmod p}$ with the $p$ signs split as evenly as possible between $+1$ and $-1$; then $F_p(x,y)=s_{-y}m_y$ with $m_y$ the occupancy of the class $-y$ in $J_p$, and $\operatorname{Var}_y F_p(x,\cdot)\asymp(KH/p)^2$. This says nothing about a saving that uses the actual law of $\bX_H$. The expansion methods of \cite{HR22,Pil23,TT25} give much stronger quantitative bounds in their respective ranges, and \cite[Corollary 1.5]{HR22} already gives a bound in terms of the window ratio alone for $\lambda(n)\lambda(n+1)$, with fixed reduced affine pairs covered by the note at the end of \S8.3; what we have not tracked by these methods is the joint dependence on growing affine coefficients and every averaging window established here.
\item \emph{Rates and smaller windows.} The rate $(\fh/\logk4\omega)^{1/2}$ is the output of this argument; letting $\eps$ depend on $H$ does not change it. Theorem~\ref{thm:A} is a statement about $\sup_{x\ge\omega}$, so no separate simultaneous-window argument is needed. Below the threshold associated with a given $\eps$, the rate form \eqref{eq:rate} may still give cancellation at a larger tolerance when $\logk4\omega$ is sufficiently large compared with $\fh$. It does not automatically preserve the original tolerance throughout the interval from $\log A$ to $A$.
\item \emph{Sign patterns.} Corollary~\ref{cor:signs} and Theorem~\ref{thm:C} give each of the four sign pairs $(\lambda(L_1(n)),\lambda(L_2(n)))$ the frequency $\frac14+o(1)$ in every window once $\fh=o(\logk4\omega)$, or once $\fh_s=o(\logk4\omega)$ and $\log(2+|m|)=o(\log\omega)$; for the correlation $\lambda(n+b)\lambda(n+b+h)$ this allows $|h|=o(\logk4\omega)$ together with $|b|=\omega^{o(1)}$. Tao--Ter\"av\"ainen's odd-order theorem \cite{TT18} gives, for three fixed pairwise admissible forms, each of the eight sign triples frequency $\frac18+o(1)$ in every window; the seventh ingredient there, the triple correlation, is controlled through the Gowers uniformity of $\Lambda$, and \cite[Remark 1.9]{TT25b} sketches how the quantitative Gowers bounds of that paper should give a triply logarithmic saving for fixed odd-order logarithmic Chowla correlations. We do not know what a growing-height, every-window version would cost, since the coefficient dependence in the generalised von Neumann step and in the comparison of primes with rough numbers has not been tracked, and we make no claim. Even-order correlations of order four remain open.
\item \emph{General $g_1,g_2$.} We make no claim for general multiplicative functions. The natural route to an analogue of Theorem~\ref{thm:A} for completely multiplicative unimodular $g_1,g_2$ with $g_1$ non-pretentious in the sense \eqref{eq:np} is to use \cite[Theorem 1.7]{MRT15} in place of Theorem 1.3, to discretise the blocks, and to allow complex coefficients in the restriction step; the bookkeeping of the level, of the change of scale from $x$ to $[x/2\omega,2x]$, and of the discretisation has not been carried out here, and the reduction from merely multiplicative to completely multiplicative functions \cite[Prop.\ 2.2]{Tao16} would have to be made quantitative as well.
\item \emph{Consumers.} A Type II estimate at $M,N\asymp X^{1/2}$ would need correlations of height $\asymp M$ with cancellation at a fixed dyadic scale; both the coefficient range and the natural-density averaging are out of reach of any $\expk4$-type threshold, and the second does not follow from the logarithmic estimates proved here. A sieve for prime tuples needs cancellation against sieve weights relative to the sifted mass. An intermediate target with content is $\sup_{\fh\le(\log\omega)^\beta}|\cdot|\ll(\log\omega)^{1-\gamma}$ for all windows, for some fixed $\beta,\gamma>0$.
\end{enumerate}

\appendix
\section{The non-pretentiousness level of \texorpdfstring{$\lambda$}{lambda}}\label{app:N}

Put $\ell:=\log\log x$, $\sigma_x:=1+1/\log x=1+e^{-\ell}$, and for $\chi\bmod q$, $t\in\R$,
\[
D_x(\chi,t):=\sum_{p\le x}\frac{1+\Re\bigl(\chi(p)p^{-it}\bigr)}p ,
\]
so that \eqref{eq:np} for $\lambda$ reads $D_x(\bar\chi,t)\ge A$, equivalently $D_x(\chi,-t)\ge A$; since the family $\{(\chi,t)\}$ is closed under conjugation of the character and under change of sign of $t$, $\cA(x)$ is the set of $A\ge1$ with $D_x(\chi,t)\ge A$ for all $q\le A$, $\chi\bmod q$, $|t|\le Ax$, and $N_\lambda(x)=\sup(\{0\}\cup\cA(x))$. $\cA(x)$ is downward closed in $[1,\infty)$, so $A\in\cA(x)$ implies $N_\lambda(x)\ge A$.

\begin{proof}[Proof of Theorem~\ref{thm:N}(i)]
Let $A\in\cA(x)$. With $q=1$, average over $t\in[0,x]$ (all $|t|\le Ax$):
\[
A\le\frac1x\int_0^xD_x(1,t)\,dt=\sum_{p\le x}\frac1p+\sum_{p\le x}\frac{\sin(x\log p)}{x\,p\log p}\le\Bigl(1+\frac1{x\log2}\Bigr)\sum_{p\le x}\frac1p .
\]
For $x\ge286$, $\sum_{p\le x}1/p\le\ell+B_M+1/(2\log^2x)\le\ell+0.278$, so $A\le\ell+0.278+(\ell+0.278)/(x\log2)\le\ell+\frac12$ (the last factor is $\le1.006$ at $x=286$ and decreasing).
\end{proof}

\subsection*{Preliminaries}
\begin{lemma}\label{lem:A1}
For $\ell\ge64$, any $\chi\bmod q$ and $t\in\R$, with $\log L$ the Euler-product logarithm on $\sigma>1$: $|\sum_{p\le x}\chi(p)p^{-1-it}-\log L(\sigma_x+it,\chi)|\le1.3891$, hence
\begin{equation}\label{eq:A1}
D_x(\chi,t)\ge\ell+\log|L(\sigma_x+it,\chi)|-1.13 .
\end{equation}
\end{lemma}

\emph{Proof.} Write
\[
\begin{aligned}
D_x(\chi,t)&=\sum_{p\le x}1/p+\Re\sum_{p\le x}\chi(p)p^{-1-it},\\
\log L(s,\chi)&=\sum_p\sum_{k\ge1}\chi(p)^kp^{-ks}/k.
\end{aligned}
\]
The prime powers $k\ge2$ contribute at most $\sum_p\sum_{k\ge2}(kp^k)^{-1}=\sum_p(-\log(1-1/p)-1/p)=0.3157\ldots<0.316$ (sum over $p\le10^4$, and bound the rest by $\frac12\sum_{n>10^4}1/(n(n-1))=1/20000$).

The smoothing $\sum_{p\le x}p^{-1}(1-p^{-1/\log x})$: with $\mathcal L_x:=\log x$, $\kappa(u):=(1-e^{-u/\mathcal L_x})/u$, which is decreasing on $u>0$, and $A(u):=\sum_{p\le e^u}(\log p)/p\le u$ (Rosser--Schoenfeld), partial summation gives $\sum_{p\le x}\frac{\log p}p\kappa(\log p)=\kappa(\mathcal L_x)A(\mathcal L_x)-\int_0^{\mathcal L_x}A(u)\kappa'(u)\,du\le \mathcal L_x\kappa(\mathcal L_x)-\int_0^{\mathcal L_x}u\kappa'(u)\,du=\int_0^{\mathcal L_x}\kappa(u)\,du=\int_0^1\frac{1-e^{-v}}v\,dv=0.7965\ldots<0.7966$.

The tail $\sum_{p>x}p^{-\sigma_x}\le1.26\,\sigma_x\int_x^\infty\frac{dy}{y^{\sigma_x}\log y}=1.26\,\sigma_xE_1(1)\le1.26(1+e^{-64})(0.21939)<0.2765$.

Finally $\sum_{p\le x}1/p\ge\ell+B_M-\frac12e^{-2\ell}\ge\ell+0.261$.

Altogether $D_x(\chi,t)\ge\ell+0.261+\log|L(\sigma_x+it,\chi)|-(0.316+0.7966+0.2765)$, and $1.3891-0.261=1.1281<1.13$. $\qed$

\begin{lemma}\label{lem:A2}
For $\psi\bmod q$, $\sigma\ge1$, $u\in\R$, with $\psi$ non-principal or $|u|\ge3$:
\[
|L(\sigma+iu,\psi)|\le510\,(\log(|u|+3))^{2/3}+94\log q+240 .
\]
\end{lemma}

\emph{Proof.} \emph{Step 1 (the line $\sigma=1$, $|t|\ge3$).} Ford \cite[Theorem 1]{For02} gives, for the Hurwitz zeta function, $|\zeta(1+it,v)-v^{-1-it}|\le76.2(\log t)^{2/3}$ for $t\ge3$, $0<v\le1$ (and by conjugation for $t\le-3$). For $\sigma>1$, grouping $n$ by residue class, $L(s,\chi)=q^{-s}\sum_{a=1}^q\chi(a)[\zeta(s,a/q)-(a/q)^{-s}]+\sum_{a=1}^q\chi(a)a^{-s}$; both sides are meromorphic for $\sigma>0$ with at most a pole at $s=1$, so the identity holds at $s=1+it$, $t\ne0$. Hence for every $\chi\bmod q$ and $|t|\ge3$,
\begin{equation}\label{eq:A2a}
|L(1+it,\chi)|\le76.2(\log|t|)^{2/3}+\sum_{a\le q}\frac1a\le76.2(\log|t|)^{2/3}+1+\log q .
\end{equation}
(The same bound with $1.92\log q$ in place of $1+\log q$ is \cite[Lemma 5.1]{Kha24} at $\sigma=1$.)
\emph{Step 2 (crude bounds).} If $\sigma\ge2$ then $|L(s,\psi)|\le\zeta(2)<2$. Let $1\le\sigma\le2$. For non-principal $\psi$, partial summation with $|\sum_{n\le v}\psi(n)|\le q$ gives $|L(s,\psi)|\le\sum_{n\le q}1/n+|s|q^{1-\sigma}/\sigma\le1+\log q+|s|$. For $\psi=\chi_0$, $L(s,\chi_0)=\zeta(s)\prod_{p\mid q}(1-p^{-s})$ with $|\prod_{p\mid q}(1-p^{-s})|\le\sum_{d\mid q}1/d\le1+\log q$, and $|\zeta(s)-1/(s-1)|\le1+|s|$ for $\sigma\ge1$, $s\ne1$ (from $\zeta(s)=s/(s-1)-s\int_1^\infty\{v\}v^{-s-1}dv$), so $|L(s,\chi_0)|\le(|s-1|^{-1}+1+|s|)(1+\log q)$. In particular, if $\psi$ is non-principal and $|u|\le3$ then $|s|\le\sqrt{13}$ and the lemma holds with room to spare.
\emph{Step 3 (three lines, $|u|\ge3$, $1\le\sigma\le2$).} Fix $t_0:=u$, $\ell_0:=\log(|t_0|+3)>1$, and $F(s):=(s-1)L(s,\psi)e^{(s-it_0)^2}$, which is entire (the factor $s-1$ removes the pole of $L(s,\chi_0)$) and bounded and continuous on $1\le\Re s\le2$, since $|e^{(s-it_0)^2}|=e^{\sigma^2-(t-t_0)^2}$ and $(s-1)L(s,\psi)$ grows at most polynomially in $|t|$ by Step 2. Put $w:=|t-t_0|$. On $\Re s=1$, for all real $t$,
\[
|t|\,|L(1+it,\psi)|\le(1+|t|)\bigl[76.2(\log(|t|+3))^{2/3}+14(1+\log q)\bigr],
\]
by \eqref{eq:A2a} for $|t|\ge3$ and by Step 2 for $|t|\le3$ (non-principal: $|t||L|\le3(1+\log q+3.17)\le14(1+\log q)$; principal: $|t||L|\le(1+3\cdot4.17)(1+\log q)\le14(1+\log q)$). Since $1+|t|\le(1+|t_0|)(1+w)$, $\log(|t|+3)\le\ell_0+\log(1+w)$ and $(x+y)^{2/3}\le x^{2/3}+y^{2/3}$,
\[
\begin{gathered}
\sup_{\Re s=1}|F(s)|\le K_0(1+|t_0|)[76.2\,\ell_0^{2/3}+14(1+\log q)],\\
K_0:=\sup_{w\ge0}(1+w)(1+(\log(1+w))^{2/3})e^{1-w^2}.
\end{gathered}
\]
An elementary bound suffices. Since $\log(1+w)\le w$ and
$w^{2/3}\le(1+2w)/3$, we have
\[
\begin{split}
K_0&\le\frac{2e}{3}\sup_{w\ge0}(w^2+3w+2)e^{-w^2}\\
&\le\frac23+\sqrt{2e}+\frac{4e}{3}
<\frac{41}{6}<7.
\end{split}
\]
Here we used $\sup w^2e^{-w^2}=1/e$,
$\sup we^{-w^2}=1/\sqrt{2e}$, and $8/3<e<11/4$.
On $\Re s=2$, $|L|\le\zeta(2)<2$ and
$|s-1|\le(1+|t_0|)(1+w)$, so
\[
\sup_{\Re s=2}|F(s)|
\le2e^4(1+|t_0|)\left(1+\frac1{\sqrt{2e}}\right)
<243(1+|t_0|).
\]
The three-lines theorem and $e^{\sigma^2}\ge e$ therefore give
\[
\begin{split}
|L(\sigma+it_0,\psi)|
&\le\frac4{3e}\max\{7[76.2\ell_0^{2/3}+14(1+\log q)],243\}\\
&\le267\ell_0^{2/3}+49\log q+171\\
&\le510\ell_0^{2/3}+94\log q+240.
\end{split}
\qquad\qed
\]

\begin{lemma}[Harnack transfer]\label{lem:A3}
Let $\chi\bmod q$, $t\in\R$, $0<\eta<1$, $r>0$ with $\eta r\le1/10$, and suppose $L(s,\chi)$ is analytic and zero-free on the closed disk of centre $z_0:=1+\eta r+it$ and radius $r$, with $\log|L|\le M$ there. If $e^{-\ell}\le\eta r$ then
\begin{equation}\label{eq:harnack}
\log|L(\sigma_x+it,\chi)|\ge-c_\eta\log\frac1{\eta r}-(c_\eta-1)M,\qquad c_\eta:=\frac{1+\eta}{1-\eta}.
\end{equation}
If moreover $\eta r\le1/1000$, the right side of \eqref{eq:harnack} may be increased by $c_\eta\log\frac85$.
\end{lemma}

\emph{Proof.} On the disk, $u:=M-\log|L|$ is nonnegative and harmonic (as $\log|L|$ is harmonic where $L$ is analytic and zero-free; no branch of $\log L$ is needed). At the centre, with $v:=\eta r$, the Euler product gives $|L(z_0,\chi)|\ge\prod_p(1+p^{-1-v})^{-1}=\zeta(2+2v)/\zeta(1+v)\ge v$, since $\zeta(1+v)\le1+1/v=(1+v)/v$ and $\zeta(2+2v)\ge1+2^{-2-2v}\ge1+v$ for $0<v\le1/10$; so $u(z_0)\le M+\log(1/(\eta r))$. If $v\le1/1000$ then $\zeta(2+2v)/\zeta(1+v)\ge\frac{v}{1+v}\zeta(2.002)\ge\frac v{1.001}\sum_{n\le25}n^{-2.002}=1.6025\ldots v>\frac85v$, which gives the last claim. The point $z:=\sigma_x+it$ satisfies $|z-z_0|=|e^{-\ell}-\eta r|\le\eta r$, so Harnack's inequality gives $u(z)\le\frac{r+\eta r}{r-\eta r}u(z_0)=c_\eta u(z_0)$, i.e.\ $\log|L(z)|\ge M-c_\eta(M+\log\frac1{\eta r})$. $\qed$

Throughout the proof of the lower bound in (ii), $\ell\ge64$, $q\le\ell$ and $|t|\le x\ell$; thus $T:=|t|+1\le x(\ell+1)$, $\log T\le e^\ell+\log(\ell+1)\le1.001\,e^\ell$; when $|t|\ge12$, also $\log\log T\le\ell+1$. The Harnack parameter is $\eta_H:=3/(8\ell)$ for $|t|\ge12$ and $\eta_L:=1/\ell$ for $|t|<12$, with $c_H:=(\ell+\frac38)/(\ell-\frac38)$ and $c_L:=(\ell+1)/(\ell-1)$; in every application of Lemma~\ref{lem:A3} below $\eta r\le1/10$ holds trivially.

\subsection*{High frequencies: \texorpdfstring{$|t|\ge12$}{|t| >= 12}}
Let $\hat q:=\max(q,3)$, $T:=|t|+1$, and
\begin{equation}\label{eq:N4}
D:=10.5\log\hat q+61.5\,(\log T)^{2/3}(\log\log T)^{1/3},\qquad r:=\frac1D,\qquad\eta:=\eta_H=\frac3{8\ell},\qquad c:=c_H .
\end{equation}
By Khale \cite[Theorem 1.1]{Kha24}, for $q\ge3$ the function $L(s,\chi)$ has no zeros in the region
\[\sigma\ge1-\frac1{10.5\log q+61.5(\log|\Im s|)^{2/3}(\log\log|\Im s|)^{1/3}},\qquad|\Im s|\ge10 .\] Every point of the closed disk in Lemma~\ref{lem:A3} has $11\le|t|-r\le|\Im s|\le|t|+r\le T$ and $\Re s\ge1+\eta r-r>1-1/D$, and $D$ is at least the corresponding denominator at each such point (the denominator is increasing in $|\Im s|\le T$), so the closed disk is zero-free. For $q\in\{1,2\}$, $L(s,\chi)$ is $\zeta(s)$ times Euler factors without zeros, and $\zeta$ is zero-free there since $L(s,\chi_0^{(3)})=\zeta(s)(1-3^{-s})$ is.

\emph{The bound $M$.} On the left half of the disk, $1-r\le\sigma\le1$: by \cite[Lemma 5.1]{Kha24} and Ford's constants, $|L(s,\chi)|\le76.2\,q^{1-\sigma}|\Im s|^{4.45(1-\sigma)^{3/2}}(\log|\Im s|)^{2/3}+1.92\frac{q^{1-\sigma}-1}{1-\sigma}$; here $q^{1-\sigma}\le q^{1/D}\le e^{2/21}$ (as $D\ge10.5\log\hat q$), so $1.92(q^{1-\sigma}-1)/(1-\sigma)\le2.02\log q$; and $4.45\,r^{3/2}\log T=4.45\,D^{-3/2}\log T<0.01$, because
\[
\frac{D^{3/2}}{\log T}\ge61.5^{3/2}(\log\log T)^{1/2}\ge61.5^{3/2}(\log\log13)^{1/2}>468 .
\]
So $|L|\le76.2e^{2/21+0.01}(\log T)^{2/3}+2.02\log q\le135\,e^{2\ell/3}+2.02\log\ell$, using $\log T\le1.001e^\ell$. For $q\in\{1,2\}$ the same bound holds for $\zeta$ (Ford's constants with $q=1$) and the Euler factor multiplying $\zeta$ has modulus $\le1.51$ on the disk, so $|L|\le1.51\cdot85e^{2\ell/3}\le135e^{2\ell/3}$ again. On the right half, $\sigma\ge1$: Lemma~\ref{lem:A2} with $|u|\ge11$ and $\log(|u|+3)\le\log(2T)\le1.01e^\ell$ gives $|L|\le510(1.01e^\ell)^{2/3}+94\log\ell+240\le514e^{2\ell/3}+94\log\ell+240$. Hence on the whole disk
\begin{equation}\label{eq:N5}
|L(s,\chi)|\le900\,e^{2\ell/3},\qquad\text{so } M:=\log900+\tfrac23\ell;\qquad\text{and}\qquad D\le63\,e^{2\ell/3}(\ell+1)^{1/3}
\end{equation}
(for the latter, $\log T\le1.001e^\ell$, $\log\log T\le\ell+1$, $61.5\cdot1.001^{2/3}<61.55$, and $10.5\log\ell\le e^{2\ell/3}$ for $\ell\ge64$). The hypothesis $e^{-\ell}\le\eta r=3/(8\ell D)$ of Lemma~\ref{lem:A3} holds since $\frac83\ell D\le168\,\ell(\ell+1)^{1/3}e^{2\ell/3}\le e^{\ell}$ for $\ell\ge64$ (the ratio of the two sides is $<10^{-4}$ at $\ell=64$, and its logarithmic derivative $1/\ell+1/(3(\ell+1))-1/3$ is negative).

\emph{Conclusion.} Since $D\ge10.5\log3+61.5(\log13)^{2/3}(\log\log13)^{1/3}>124$, we have $\eta r\le3/(512\cdot124)$, which is less than $1/1000$; so Lemma~\ref{lem:A3} applies with the improved centre value. By \eqref{eq:A1}, \eqref{eq:harnack}, \eqref{eq:N5} and
\[
\log\frac1{\eta r}-\log\frac85=\log\frac{5\ell D}3\le\log(105\ell)+\frac23\ell+\frac13\log(\ell+1),
\]
we get
\begin{equation}\label{eq:N6}
D_x(\chi,t)\ge\ell-1.13-c\Bigl[\log(105\ell)+\frac23\ell+\frac13\log(\ell+1)\Bigr]-(c-1)\Bigl[\log900+\frac23\ell\Bigr]
=\frac\ell3-\frac43\log\ell-C_{\rm high}(\ell),
\end{equation}
where, using $\ell-\frac23c\ell-\frac23(c-1)\ell=\frac\ell3-\frac43(c-1)\ell$ and $\frac43(c-1)\ell=\frac{8\ell}{8\ell-3}$,
\[
\begin{aligned}
C_{\rm high}(\ell)&:=\frac{8\ell}{8\ell-3}+c\log105+(c-1)\frac43\log\ell\\
&\quad+\frac c3\log\Bigl(1+\frac1\ell\Bigr)+(c-1)\log900+1.13,
\qquad c=c_H=\frac{\ell+\frac38}{\ell-\frac38}.
\end{aligned}
\]
Each nonconstant term is decreasing in $\ell\ge64$ (for
$(c-1)\log\ell=6\log\ell/(8\ell-3)$ the derivative has the sign of
$8-3/\ell-8\log\ell<0$). Exact rational enclosures give
\[
6.99549<C_{\rm high}(64)<6.99550<7.
\]
Hence $D_x(\chi,t)\ge\ell/3-\frac43\log\ell-7$ for $|t|\ge12$.
For reproducibility, the logarithms in this enclosure and the resonator
enclosure below can be bounded by reducing each positive argument to
$2^j v$, $1\le v\le2$, and using, with $u=(v-1)/(v+1)$,
\[
0\le\log v-2\sum_{k=0}^{31}\frac{u^{2k+1}}{2k+1}
\le\frac{2u^{65}}{65(1-u^2)}.
\]
The same formula bounds $\log2$; rational interval arithmetic propagates
these bounds through the displayed expressions.

\subsection*{Low frequencies: \texorpdfstring{$|t|<12$}{|t| < 12}}
Put
\begin{equation}\label{eq:N7}
c_0:=\frac{0.96\,e^{-2}}{\sqrt\ell\,(\log\ell)^2},\qquad d_0:=(\log\ell+3)^2,\qquad t_c:=\frac{c_0}{4d_0},\qquad r_0:=\frac1{40\log(20\ell)}.
\end{equation}
For $\ell\ge64$ one has $t_c/4\le r_0$ and $e^{-\ell}\le t_c/(4\ell)$ (both hold at $\ell=64$ and the ratios are monotone; the logarithmic derivatives of $t_c/(4r_0)$ and of $4\ell e^{-\ell}/t_c$ are negative for $\ell\ge64$).

By McCurley's zero-free region \cite{McC84}, in the formulation of Kadiri \cite[Theorem 2.2]{Kad18}, for $q\ge1$ the function $L(s,\chi)$ has at most one zero in the region $\sigma\ge1-1/(R\log\max\{q,\,q|\Im s|,\,10\})$, $R=9.645908801$, and such a zero is real, simple, and occurs only for real non-principal $\chi$. Since $q\le\ell$ and $|\Im s|\le13$ on all disks below, $\max\{q,q|\Im s|,10\}\le13\ell$ and $R\log(13\ell)\le10\log(20\ell)$, so the region contains the part of the half-plane $\sigma\ge1-1/(10\log(20\ell))$ lying in the strip $|\Im s|\le13$, and in particular every disk below (all of which lie in $\sigma\ge1-r_0$, $|\Im s|\le13$).

In Cases L1--L3 we use $\eta:=\eta_L=1/\ell$ and $c:=c_L=(\ell+1)/(\ell-1)\le65/63$.

\emph{Case L1: $\chi$ non-real, or $\chi$ real with $|t|\ge t_c$.} Take $r:=r_0$ if $\chi$ is non-real, and $r:=\min(r_0,|t|/4)$ otherwise; in both cases $r\ge t_c/4$ and $\eta r\ge e^{-\ell}$. The disk of Lemma~\ref{lem:A3} lies in McCurley's region; for non-real $\chi$ it contains no zero, and for real $\chi$ with $|t|\ge t_c$ its points have $|\Im s|\ge|t|-r\ge3|t|/4>0$, so it avoids the real axis and hence both the possible real zero and the pole of $L(s,\chi_0)$. On the disk, $|s|\le14$ and $\sigma\ge1-r_0$, so for non-principal $\chi$ partial summation gives $|L(s,\chi)|\le q^{r_0}(1+\log q)+14q^{r_0}/(1-r_0)\le2(16+\log\ell)$ (as $q^{r_0}\le e^{1/40}$), while for $\chi_0$, $|\zeta(s)|\le|s-1|^{-1}+16$, $|s-1|\ge3|t|/4\ge3t_c/4$, and, since $\sigma\ge1-r_0$,
\[
\Bigl|\prod_{p\mid q}(1-p^{-s})\Bigr|\le\prod_{p\mid q}(1+p^{-1+r_0})=\sum_{d\mid\operatorname{rad}(q)}d^{-1+r_0}\le q^{r_0}\sum_{d\mid q}\frac1d\le e^{1/40}(1+\log q)<3(1+\log\ell),
\]
so $|L(s,\chi_0)|\le(4/t_c+48)(1+\log\ell)$. A common bound is $M:=\log U_{\rm low}$, $U_{\rm low}:=4(1/t_c+20)(1+\log\ell)$ (note $48\le80$). Lemma~\ref{lem:A3} with $r\ge t_c/4$, $1/(\eta r)\le4\ell/t_c\le8\ell/t_c$, gives
\[
D_x(\chi,t)\ge\ell-c\log\frac{8\ell}{t_c}-(c-1)\log U_{\rm low}-1.13 .
\]
For $\ell\ge64$: $\log\log\ell\le\frac12\log\ell$ and $\log\ell+3\le2\log\ell$, so $\log(1/t_c)\le\log\frac{16e^2}{0.96}+\frac52\log\ell$; and $t_c\le\frac1{20}$, $1+\log\ell\le2\log\ell$ give $\log U_{\rm low}\le\log16+\log\frac{16e^2}{0.96}+3\log\ell$. With $c=c_L\le65/63$ the coefficient of $\log\ell$ is at most $\frac72c+3(c-1)=\frac{467}{126}<4$ and the additive constant is $c(\log8+\log\frac{16e^2}{0.96})+(c-1)(\log16+\log\frac{16e^2}{0.96})+1.13=8.48\ldots<9$. Hence
\begin{equation}\label{eq:N10}
D_x(\chi,t)\ge\ell-4\log\ell-9\ \ >\ \frac\ell3-\frac43\log\ell-7 .
\end{equation}

\emph{Case L2: $\chi$ real non-principal, $|t|\le t_c$.} Here $q\ge3$. Let $\chi^*\bmod q^*$ induce $\chi$, $q^*>1$. Let $d$ be the fundamental discriminant with $\chi^*=\chi_d$, so $q^*=|d|$, let $h_d$ be the (ordinary) ideal class number and $w_d$ the number of roots of unity of $\mathbb Q(\sqrt d)$, and for $d>0$ let $\varepsilon_d>1$ generate the unit group modulo torsion (its norm may be $-1$). The analytic class number formula and $\zeta_{\mathbb Q(\sqrt d)}(s)=\zeta(s)L(s,\chi_d)$ \cite[Lecture 19, Theorems 19.12 and 19.15; Lecture 15, Example 15.17]{Suth21} give $L(1,\chi^*)\ge0.96/\sqrt{q^*}$: for $d<0$, $L(1,\chi^*)=2\pi h_d/(w_d\sqrt{|d|})\ge\pi/(3\sqrt{|d|})$; for $d>0$, $L(1,\chi^*)=2h_d\log\varepsilon_d/\sqrt d\ge2\log\frac{1+\sqrt5}2/\sqrt d>0.96/\sqrt d$ (a unit $\varepsilon>1$ has integral trace and norm $\pm1$, so $\varepsilon\ge\frac{1+\sqrt5}2$). Moreover, $L(1,\chi)=L(1,\chi^*)\prod_{p\mid q,\,p\nmid q^*}(1-\chi^*(p)/p)\ge\frac{0.96}{\sqrt q}\prod_{p\mid q}(1-p^{-1})$. Since $\prod_{p\mid q}(1-p^{-1})^{-1}\le\zeta(2)\prod_{p\mid q}(1+p^{-1})\le\zeta(2)\sum_{d\mid q}d^{-1}\le\zeta(2)(1+\log q)\le e^2(\log q)^2$ for $q\ge3$, we get $L(1,\chi)\ge c_0$ (as $q\le\ell$). For $s$ with $1\le\sigma\le2$, $|s|\le2$, partial summation as in Lemma~\ref{lem:A2} gives $|L'(s,\chi)|\le\sum_{n\le q}\frac{\log n}n+q\int_q^\infty\frac{1+2\log v}{v^2}dv\le(\log q)(1+\log q)+3+2\log q\le(\log q+3)^2\le d_0$. On the segment from $1$ to $\sigma_x+it$ (length $\le e^{-\ell}+t_c\le2t_c$ by \eqref{eq:N7}), $|L(\sigma_x+it,\chi)-L(1,\chi)|\le2d_0t_c=c_0/2$, so $|L(\sigma_x+it,\chi)|\ge c_0/2$ and by \eqref{eq:A1}
\[
D_x(\chi,t)\ge\ell+\log\frac{0.48e^{-2}}{\sqrt\ell(\log\ell)^2}-1.13\ \ge\ \ell-\frac12\log\ell-2\log\log\ell-3.87\ \ >\ \frac\ell3-\frac43\log\ell-7 .
\]

\emph{Case L3: $\chi=\chi_0$, $|t|\le t_c$.} $|\zeta(\sigma_x+it)|\ge|s-1|^{-1}-1-|s|\ge(2t_c)^{-1}-3\ge1$, and $|\prod_{p\mid q}(1-p^{-s})|\ge\prod_{p\mid q}(1-p^{-1})\ge e^{-2}(\log q)^{-2}$ for $q\ge3$ (and $\ge\frac12$ for $q\le2$), so $D_x(\chi_0,t)\ge\ell-2-2\log\log\ell-1.13>\ell/3-\frac43\log\ell-7$.

\begin{proof}[Proof of Theorem~\ref{thm:N}(ii), lower bound]
The cases $|t|\ge12$, L1, L2, L3 exhaust all $\chi\bmod q$ with $q\le\ell$ and $|t|\le x\ell$, and in each $D_x(\chi,t)\ge A_1:=\ell/3-\frac43\log\ell-7$ (in Cases L1--L3 the bounds obtained are larger than $A_1$ for $\ell\ge64$: e.g.\ $\ell-4\log\ell-9-A_1=\frac23\ell-\frac83\log\ell-2$ is positive and increasing). For $\ell\ge64$, $1\le A_1<\ell$ ($A_1=8.788\ldots$ at $\ell=64$, increasing), so the ranges $q\le A_1$, $|t|\le A_1x$ are covered, $A_1\in\cA(x)$, and $N_\lambda(x)\ge A_1$. The unrounded bound is $\min(\text{\eqref{eq:N6}},\text{\eqref{eq:N10}},\dots)=\eqref{eq:N6}$ for $\ell\ge64$.
\end{proof}

\subsection*{The upper bound: a finite variable-box resonator}
Write $H(v):=\sum_{p\le v}1/p$ and
$P_x(t):=\sum_{p\le x}p^{-1-it}$, so
$D_x(1,t)=H(x)+\Re P_x(t)$.

\begin{proof}[Proof of Theorem~\ref{thm:N}(ii), upper bound]
Let $\ell\ge64$, $\mathcal L:=\log x=e^\ell$, $z:=\mathcal L/2$,
$\ell_z:=\log z=\ell-\log2$, and, for $p\le z$, put
$k_p:=\lfloor\sqrt{z/p}\rfloor$. Define
\[
\mathcal S:=\left\{\prod_{p\le z}p^{j_p}:0\le j_p\le k_p\right\},
\quad N:=\prod_{p\le z}p^{k_p},\quad
J:=|\mathcal S|=\prod_{p\le z}(k_p+1),\quad
R(t):=\sum_{n\in\mathcal S}\lambda(n)n^{it}.
\]
Since $k_p=\#\{j\ge1:\ j^2p\le z\}$, we have $\log N=\sum_{j\ge1}\theta(z/j^2)$ and, as $k+1\le2^k$ for integers $k\ge1$, $\log J\le\log2\sum_{p\le z}k_p=\log2\sum_{j\ge1}\pi(z/j^2)$. The prime bounds of Rosser--Schoenfeld give
\[
\log N=\sum_{j\ge1}\theta(z/j^2)
\le1.02\zeta(2)z<0.84\mathcal L.
\]
For $j\le z^{1/4}$ use
$\log(z/j^2)\ge \ell_z/2$ and $\pi(u)\le1.26u/\log u$;
for larger $j$ use $\pi(u)\le u$ and
$\sum_{j>z^{1/4}}j^{-2}\le2z^{-1/4}$. Hence
\[
\frac{\log J}{\mathcal L}
\le\log2\left(\frac{1.26\zeta(2)}{\ell_z}+e^{-\ell_z/4}\right)
<0.03.
\]
The expression is decreasing in $\ell$, and its value at $64$ is
$0.02269\ldots$. Thus $N\le x^{0.84}$ and $J\le x^{0.03}$.

Every $n\in\mathcal S$ is at most $N$. If positive integers $n\ne h$
have $\min(n,h)\le N$, then
$|\log(h/n)|\ge\log(1+1/N)\ge1/(2N)$, whence
\[
\left|\frac1x\int_0^x e^{it\log(h/n)}\,dt\right|\le\frac{4N}x.
\]
Put
\[
\mathcal B:=\frac1x\int_0^x|R(t)|^2dt,\qquad
\mathcal A:=\frac1x\int_0^xP_x(t)|R(t)|^2dt,
\quad a_0:=\sum_{p\le z}\frac{k_p}{(k_p+1)p}.
\]
The diagonal in $\mathcal B$ is $J$. For a fixed $p\le z$, the
pairs $n=ph$ with $n,h\in\mathcal S$ number $Jk_p/(k_p+1)$,
and their signs are $\lambda(ph)\lambda(h)=-1$.
There is no diagonal for $p>z$. Every off-diagonal integral in
$\mathcal A$ is at most $4N/x$, because the smaller of $n$ and $ph$
is at most $N$. Thus, with
$\delta:=4NJ/x\le4x^{-0.13}<1/2$,
\[
\left|\frac{\mathcal B}J-1\right|\le\delta,\qquad
\left|\frac{\Re\mathcal A}J+a_0\right|\le\delta H(x).
\]
Writing $\mathcal B/J=1+v$ and $\Re\mathcal A/J=-a_0+u$ gives
\[
\frac{\Re\mathcal A}{\mathcal B}+a_0
=\frac{u+a_0v}{1+v}\le\frac{\delta(H(x)+a_0)}{1-\delta}
\le4\delta H(x).
\]
The nonnegative weight $|R|^2$ has positive integral, and $D_x(1,t)$
is continuous. Consequently some $t\in[0,x]$ satisfies
\begin{equation}\label{eq:res}
D_x(1,t)\le H(x)-H(z)+E(z)+16x^{-0.13}H(x),
\qquad E(z):=\sum_{p\le z}\frac1{p(k_p+1)}.
\end{equation}

For $z/4<p\le z$ one has $k_p=1$ exactly. For $p\le z/4$,
$k_p+1>\sqrt{z/p}$. Hence
\[
 E(z)\le\frac12\bigl(H(z)-H(z/4)\bigr)
       +z^{-1/2}\sum_{p\le z/4}p^{-1/2}.
\]
For $u\ge e^{10}$ put $f(u):=u^{-1/2}/\log u$ and
$G(u):=\sqrt u/(\log u-5)$. Direct differentiation, with $v=\log u$, gives
\[
 G'(u)-u^{-1/2}\left(\frac1{2\log u}+\frac1{\log^2u}\right)
 =u^{-1/2}\frac{(v-10)(v+5)}{2v^2(v-5)^2}\ge0.
\]
Partial summation against $\theta(u)\le1.02u$, and
$\sum_{p\le e^{10}}p^{-1/2}\le\sum_{n\le e^{10}}n^{-1/2}\le2e^5$, give,
for $V\ge e^{10}$,
\[
 \sum_{p\le V}p^{-1/2}
 \le2e^5+1.02\sqrt V\left(\frac1{\log V}+\frac1{\log V-5}\right).
\]
Here the dropped boundary terms are $-\theta(e^{10})f(e^{10})$ and
$-1.02G(e^{10})$, both negative.
Put $d_z:=\log(z/4)=\ell-3\log2$; then $d_z\ge10$.
The two Mertens inequalities and the last display yield
\[
 E(z)\le\frac12\log\frac{\ell_z}{d_z}
 +\frac1{4\ell_z^2}+\frac1{4d_z^2}
 +2e^{5-\ell_z/2}
 +0.51\left(\frac1{d_z}+\frac1{d_z-5}\right).
\]
Also
\[
 H(x)-H(z)\le\ell-\log\ell_z+\frac{e^{-2\ell}}2+\frac1{2\ell_z^2},
 \qquad H(x)\le\ell+0.278.
\]
Therefore \eqref{eq:res} is at most $\ell-\log\ell+F(\ell)$, where
\begin{align*}
 F(\ell):={}&-\log\left(1-\frac{\log2}{\ell}\right)
 +0.51\left(\frac1{\ell-3\log2}+\frac1{\ell-3\log2-5}\right)\\
 &+\frac12\log\frac{\ell-\log2}{\ell-3\log2}
 +2e^{5-(\ell-\log2)/2}+\frac{e^{-2\ell}}2\\
 &+\frac3{4(\ell-\log2)^2}+\frac1{4(\ell-3\log2)^2}
 +16(\ell+0.278)e^{-0.13e^\ell}.
\end{align*}
Every term in $\ell F(\ell)$ is decreasing for $\ell\ge64$.
For the first logarithm, writing $u=(\log2)/\ell$, the derivative is
$-\log(1-u)-u/(1-u)<0$. For the other logarithm use
\[
 \ell\log\frac{\ell-\log2}{\ell-3\log2}
 =\int_{\log2}^{3\log2}\frac{\ell}{\ell-r}\,dr,
\]
whose integrand is decreasing in $\ell$. The rational and exponential terms
are decreasing directly; the logarithmic derivative of the final term is
$1/\ell+1/(\ell+0.278)-0.13e^\ell<0$.
At $\ell=64$, the sum of all but the final contribution to $\ell F(\ell)$
lies strictly between $2.52215$ and $2.52217$.
For this enclosure, the two ordinary exponential contributions are
positive and together are less than $2^{-19}+2^{-123}$, using
$\log2<2$ and $e>2$; the remaining terms are bounded by the logarithm
series above. The final contribution is
$16\ell(\ell+0.278)e^{-0.13e^\ell}$. At $\ell=64$ it is less than
$2^{17}e^{-512}<2^{-495}<10^{-100}$, and it decreases thereafter.
Consequently $F(\ell)<2.53/\ell$ for every $\ell\ge64$.
Every $A\in\cA(x)$ satisfies $A\le D_x(1,t)$ for this same $t$,
since $q=1\le A$ and $t\le x\le Ax$. Taking the supremum proves
the assertion; if $\cA(x)$ is empty it is immediate.
\end{proof}

\begin{remark*}
Asymptotically, $(\frac13-o(1))\ell$ is also what \cite[(1-12)]{MRT15} gives after replacing the scale $x$ by $x\ell$ (the removed tail $\sum_{x<p\le x\ell}1/p$ is $O(\log\ell/\log x)$); part (ii) makes the $o(1)$ explicit. The constant $\frac13=1-\frac23$ is the Vinogradov--Korobov exponent applied with $\log|t|\le\log x+O(\log\ell)$; the leading loss is $c\log(1/r)\approx\frac23\ell+\frac13\log\ell$ from the width of the zero-free region, while $(c-1)M=O(1)$ once $\eta\asymp1/\ell$. For $\eta=\kappa/\ell$ the $\eta$-dependent part of the constant is asymptotically $-\log\kappa+\frac83\kappa$, minimised at $\kappa=\frac38$, which is the choice made above; this optimises the single-disk Harnack argument and is not an optimality statement about $N_\lambda$. Under GRH the Harnack step is replaced by the comparison of two smoothed truncations of the Euler product (Lemma~\ref{lem:A5}), which gives part (iii); the resonator upper bound shows that the resulting $\log\ell$ loss is genuine, so that $N_\lambda(x)=\ell-\log\ell+O(1)$ under GRH.
\end{remark*}

\begin{lemma}[GRH contour framework and truncation]\label{lem:A4}
Assume GRH for Dirichlet $L$-functions. Let $\chi\bmod q$, $t\in\R$, $\mathcal Q:=q(|t|+3)$ and $y:=\max(2,\log^2\mathcal Q)$, and suppose that either $\chi$ is non-principal or $|t|\ge1$. Then, uniformly for $1\le\sigma\le2$,
\begin{equation}\label{eq:A4}
\log L(\sigma+it,\chi)=\sum_{p\le y}\frac{\chi(p)}{p^{\sigma+it}}+O(1),
\end{equation}
where $\log L$ is the Euler-product logarithm, continued to $\sigma=1$ from $\sigma>1$, and the implied constant is absolute and effective.
\end{lemma}

(Related explicit smoothed formulas for $\log|L(1,\chi)|$ and $\log|L(1,f)|$ are proved in \cite[Lemma 2.5]{LLS15} and \cite[Lemma 3.1]{Lum18}; \cite[Corollary 2.3]{Lum18} also gives bounds in the $t$-aspect. We give the argument here for the stated complex-logarithm truncation, uniformly for $1\le\sigma\le2$, including imprimitive characters and $\zeta$ at $|t|\ge1$. The $O(1)$ conclusion is not the numerical input to Theorem~\ref{thm:N}(iii); the proof supplies the contour framework and the kernel bound that Lemma~\ref{lem:A5} makes explicit.)

\emph{Proof.} Let $w(v):=1$ for $0<v\le1$, $w(v):=2-v$ for
$1<v\le2$, and $w(v):=0$ for $v>2$. Its Mellin transform is
\[
 \widehat w(z)=\frac{2^{z+1}-1}{z(z+1)}
             =\frac1z\int_1^2 u^z\,du.
\]
It has a simple pole of residue $1$ at $z=0$ and a removable singularity
at $z=-1$. We shall use the kernel bound
\begin{equation}\label{eq:kernelsharp}
 (1+|z|^2)|\widehat w(z)|\le C_w:=\frac{45(1+\sqrt2)}{41}
 <\frac83
 \qquad\left(-2\le\Re z\le-\frac12\right).
\end{equation}
In fact the bound holds throughout $\Re z\le-1/2$. Put
$U(z):=(1+|z|^2)|\widehat w(z)|$. The removable value is
$\widehat w(-1)=-\log2$, so the vector
$(\widehat w(z),\sqrt2\,z\widehat w(z),z^2\widehat w(z))$
is holomorphic on a neighbourhood of this half-plane. Its Euclidean norm
is $U$, which is therefore continuous and subharmonic.

On the boundary write $z=-1/2+it$ and $y=t^2$.
For $0\le y\le7$, the integral representation gives
\[
 U(z)\le2(\sqrt2-1)(r+1/r),\qquad r:=\sqrt{y+1/4}.
\]
Since $r+1/r$ has no interior maximum, its endpoint bounds give
$U(z)\le\max\{5(\sqrt2-1),33(\sqrt2-1)/\sqrt{29}\}$.
For the other ranges use the exact boundary identity
\[
 U(-1/2+it)=\left(1+\frac1{y+1/4}\right)
       \sqrt{3-2\sqrt2\cos(t\log2)}.
\]
If $7\le y\le10$, then $|t|\log2\le\sqrt{10}\log2<11/5<\pi$, and
\[
 \cos(t\log2)\ge\cos(11/5)
 \ge1-\frac{(11/5)^2}{2}+\frac{(11/5)^4}{24}
       -\frac{(11/5)^6}{720}>-\frac23.
\]
Thus $U(z)\le(33/29)\sqrt{3+(4/3)\sqrt2}$ in this range.
For $y\ge10$, the bound $\cos(t\log2)\ge-1$ gives
$U(z)\le45(1+\sqrt2)/41=C_w$.
The first three bounds are smaller than $C_w$, as is checked by squaring.

To pass from the boundary to the half-plane, let $R\ge13$.
On $|z|=R$, $\Re z\le-1/2$, we have
\[
 U(z)\le(1+\sqrt2)\frac{R^2+1}{R(R-1)}<C_w,
\]
because $|2^{z+1}-1|\le1+\sqrt2$, $|z+1|\ge R-1$, and
$4R^2-45R-41>0$. The subharmonic maximum principle on the bounded
domain $\{\Re z<-1/2,\ |z|<R\}$ therefore gives $U\le C_w$ there.
Taking $R$ arbitrarily large proves \eqref{eq:kernelsharp}.
Finally, $C_w<8/3$ is equivalent to $\sqrt2<193/135$ and follows by squaring.

Write $s=\sigma+it$, $1\le\sigma\le2$.

\emph{Step 1: primitive non-principal $\chi$ of conductor $q^*>1$.} Mellin inversion gives, for $y\ge2$,
\[
\sum_n\frac{\chi(n)\Lambda(n)}{n^s}w\Bigl(\frac ny\Bigr)=\frac1{2\pi i}\int_{(2)}\Bigl(-\frac{L'}L\Bigr)(s+z,\chi)\,y^z\widehat w(z)\,dz .
\]
Move the line of integration to $\Re z=\frac14-\sigma$, along rectangles with horizontal edges $\Im(s+z)=T_j^\pm$, where $T_j^+\to+\infty$ and $T_j^-\to-\infty$ are chosen at distance at least $c/\log(q^*(|T_j^\pm|+3))$ from every zero ordinate $\gamma$ (possible, for a small effective absolute $c>0$, since each unit interval of heights contains $O(\log(q^*(|T|+3)))$ ordinates).

On these edges $L'/L(s',\chi)\ll\log^2(q^*(|\Im s'|+3))$ uniformly for $\frac14\le\Re s'\le4$ (trivially for $\Re s'\ge2$, and otherwise by the partial-fraction formula $L'/L(s',\chi)=\sum_{|\gamma-\Im s'|\le1}(s'-\rho)^{-1}+O(\log(q^*(|\Im s'|+3)))$ \cite[Lemma 12.6, (12.5), p.~402]{MV07} and the local zero count \cite[Theorem 10.17, p.~351]{MV07}, $N(T+1,\chi)-N(T,\chi)\ll\log(q^*(|T|+3))$), while $|\widehat w(z)|\ll|\Im z|^{-2}$ there, so the edge integrals tend to $0$ ($s$ and $y$ being fixed; the uniformity claimed below comes from the residue and vertical-integral bounds).

The poles crossed are $z=0$, with residue $-L'/L(s,\chi)$, and $z=\rho-s$ for the nontrivial zeros $\rho=\frac12+i\gamma$ (with multiplicity), with residue $-y^{\rho-s}\widehat w(\rho-s)$; the trivial zeros have real part $\le0$ and are not crossed.

Under GRH, $|y^{\rho-s}|=y^{1/2-\sigma}$, $|\widehat w(\rho-s)|\le(8/3)/(1+(\gamma-t)^2)$ and $\sum_\rho(1+(\gamma-t)^2)^{-1}\ll\log\mathcal Q^*$, $\mathcal Q^*:=q^*(|t|+3)$ (from the local zero count). On the new line, $s+z=\frac14+iu$; by the functional equation, the partial-fraction formula at $\Re s'=\frac34$, and GRH (each $|\frac34+iu-\rho|\ge\frac14$), $L'/L(\frac14+iu,\chi)\ll\log(q^*(|u|+3))$, so this integral is $\ll y^{1/4-\sigma}\int\log(q^*(|u|+3))(1+(u-t)^2)^{-1}du\ll y^{1/4-\sigma}\log\mathcal Q^*$. Altogether, uniformly for $1\le\sigma\le2$,
\begin{equation}\label{eq:A4a}
-\frac{L'}L(s,\chi)=\sum_n\frac{\chi(n)\Lambda(n)}{n^s}w\Bigl(\frac ny\Bigr)+O\bigl(y^{1/2-\sigma}\log\mathcal Q^*\bigr).
\end{equation}
Integrate \eqref{eq:A4a} along the segment from $s$ to $2+it$ and use $\log L(2+it,\chi)=\sum_n\chi(n)\Lambda(n)n^{-2-it}/\log n$:
\[
\log L(s,\chi)=\sum_n\frac{\chi(n)\Lambda(n)}{n^s\log n}w\Bigl(\frac ny\Bigr)+\sum_{n>y}\frac{\chi(n)\Lambda(n)(1-w(n/y))}{n^{2+it}\log n}+O\Bigl(\frac{y^{1/2-\sigma}\log\mathcal Q^*}{\log y}\Bigr).
\]
The second sum is at most $\sum_{n>y}n^{-2}\le1$. In the first, the prime powers $p^k$, $k\ge2$, contribute at most $\sum_p\sum_{k\ge2}(kp^k)^{-1}\le\frac12\sum_p(p(p-1))^{-1}<0.39$, and the primes $y<p\le2y$ at most $\sum_{y<n\le2y}1/n\le1$. Hence, for all $y\ge2$ and $1\le\sigma\le2$,
\begin{equation}\label{eq:A4b}
\log L(s,\chi)=\sum_{p\le y}\frac{\chi(p)}{p^{s}}+O\Bigl(1+\frac{\log\mathcal Q^*}{\sqrt y\,\log y}\Bigr).
\end{equation}

\emph{Step 2: non-principal $\chi\bmod q$.} Let $\chi$ be induced by the primitive $\chi^*\bmod q^*$, $q^*>1$, and take $y=\max(2,\log^2\mathcal Q)$; since $\mathcal Q^*\le\mathcal Q$, the error in \eqref{eq:A4b} for $\chi^*$ is $O(1)$. Now $L(s,\chi)=L(s,\chi^*)\prod_{p\mid q,\,p\nmid q^*}(1-\chi^*(p)p^{-s})$, so $\log L(s,\chi)=\log L(s,\chi^*)-\sum_{p\mid q,\,p\nmid q^*}\sum_{k\ge1}\chi^*(p)^kp^{-ks}/k$. The terms $k=1$, $p\le y$ turn $\sum_{p\le y}\chi^*(p)p^{-s}$ into $\sum_{p\le y}\chi(p)p^{-s}$; the terms $k=1$, $p>y$ are at most $\sum_{p\mid q,\,p>y}1/p\le(\log q/\log y)/y\le1/(\sqrt y\log y)<1.1$ (as $\log q\le\sqrt y$); the terms $k\ge2$ are at most $0.39$. This gives \eqref{eq:A4}.

\emph{Step 3: principal $\chi_0\bmod q$, $|t|\ge1$.} For $\zeta$ the contour argument of Step 1 crosses one more pole, at $s+z=1$, where $-\zeta'/\zeta$ has residue $+1$; the shifted identity reads
\[
\sum_n\frac{\Lambda(n)}{n^s}w\Bigl(\frac ny\Bigr)=-\frac{\zeta'}\zeta(s)+y^{1-s}\widehat w(1-s)-\sum_\rho y^{\rho-s}\widehat w(\rho-s)+\frac1{2\pi i}\int_{\Re z=\frac14-\sigma}\Bigl(-\frac{\zeta'}\zeta\Bigr)(s+z)y^z\widehat w(z)\,dz .
\]
Since $|t|\ge1$ and $\sigma\ge1$, $|1-s|,|2-s|\ge1$ and $|\widehat w(1-s)|\le(2^{2-\sigma}+1)/(|1-s||2-s|)\le3$, so the extra term in \eqref{eq:A4a} is $\le3y^{1-\sigma}$ and its integral over $\sigma\le\Re s'\le2$ is $\le3/\log y\le5$. Hence \eqref{eq:A4b} holds for $\zeta$ with $\mathcal Q^*=|t|+3$, and the Euler factors of $L(s,\chi_0)=\zeta(s)\prod_{p\mid q}(1-p^{-s})$ are handled as in Step 2. $\qed$

\begin{lemma}[An explicit smoothed approximation under GRH]\label{lem:A5}
Assume GRH. Let $\chi\bmod q$, $t\in\R$, $\mathcal Q:=q(|t|+3)$,
and suppose that $\chi$ is non-principal or $|t|\ge1$.
With $w$ as in the proof of Lemma~\ref{lem:A4} and $\log L$ as in Lemma~\ref{lem:A4}, for $v\ge2$ put
\[
 W_v(\chi,t):=\sum_n\frac{\chi(n)\Lambda(n)}{n^{1+it}\log n}w(n/v).
\]
Then
\begin{equation}\label{eq:A5}
\begin{split}
 |\log L(1+it,\chi)-W_v(\chi,t)|
 \le{}&\frac{2\log\mathcal Q+16}{\sqrt v\log v}
       +\frac{5\log\mathcal Q+62}{v^{3/4}\log v}
       +\frac1{\log v}+\frac2v\\
     &+\frac{\log q}{v\log v}+\frac4{\sqrt v}.
\end{split}
\end{equation}
In particular the coarser, single-power estimate is
\begin{equation}\label{eq:A5compact}
 |\log L(1+it,\chi)-W_v(\chi,t)|
 \le\frac{80\log\mathcal Q}{\sqrt v\log v}
 +\frac1{\log v}+\frac2v+\frac{\log q}{v\log v}+\frac4{\sqrt v}.
\end{equation}
The sharper form \eqref{eq:A5}, not \eqref{eq:A5compact}, is used below.
\end{lemma}

\emph{Proof.} First let $L$ be primitive non-principal of conductor $q^*>1$,
or let $L=\zeta$ and put $q^*=1$; write $\mathcal Q^*:=q^*(|t|+3)$.
Let $\psi_0:=\Gamma'/\Gamma$. Euler summation and the differentiated
series for $\psi_0$ give, for $\Re z>0$,
\[
 \psi_0(z)=\log z-\frac1{2z}
 +\int_0^\infty\frac{B_1(\{u\})}{(z+u)^2}\,du,
 \qquad |B_1(\{u\})|\le\frac12.
\]
Consequently
\begin{equation}\label{eq:A5gamma}
 |\psi_0(z)-\log z|\le\frac1{\Re z},\qquad
 |\psi_0'(z)|\le\frac1{(\Re z)^2}+\frac1{\Re z}.
\end{equation}
Also, on $\Re s=2$,
\[
 |L'/L(s)|\le\sum_{n\ge2}\frac{\log n}{n^2}
 \le\frac{\log2}{4}+\frac{1+\log2}{2}<2.
\]
The completed-function identities
\cite[Corollary 10.14, (10.25), (10.29)--(10.30), p.~349]{MV07} and
\cite[Corollary 10.18, (10.37)--(10.38), p.~352]{MV07}, after taking
real parts and using $\Re B(\chi)=-\sum_\rho\Re(1/\rho)$ for the
Hadamard-product constant $B(\chi)$ and its zeta analogue, give under GRH
\[
 \sum_\rho\frac{3/2}{9/4+(\gamma-t)^2}
 =\Re\frac{L'}L(2+it)+\frac12\log\frac{q^*}{\pi}
 +\frac12\Re\psi_0\left(\frac{2+\kappa+it}{2}\right)
\]
for primitive non-principal $L$, with parity $\kappa\in\{0,1\}$.
For $\zeta$, take $q^*=1$, $\kappa=0$ and add
$\Re((2+it)^{-1}+(1+it)^{-1})$ to the right.
These right sides are at most $5/2+(1/2)\log\mathcal Q^*$ and
$4+(1/2)\log\mathcal Q^*$, respectively.
Since $(1+u^2)^{-1}\le(3/2)(3/2)/(9/4+u^2)$, we have
\begin{equation}\label{eq:A5zeros}
 \sum_\rho\frac1{1+(\gamma-t)^2}
 \le\frac34\log\mathcal Q^*+6
 \le\frac{25}{4}\log\mathcal Q^*.
\end{equation}
Zeros are counted with multiplicity, and the last inequality follows from
$3/4+6/\log3<25/4$.

Subtract the full partial-fraction identities at $1/4+iu$ and $2+iu$.
Each zero contributes at most
\[
 \frac{7/4}{\sqrt{(1/16+(\gamma-u)^2)(9/4+(\gamma-u)^2)}}
 \le\frac{14/3}{1+(\gamma-u)^2},
\]
because
\[
 (r^2+1/16)(r^2+9/4)-\frac9{64}(1+r^2)^2
 =\frac{55}{64}r^4+\frac{65}{32}r^2\ge0.
\]
For the gamma-factor difference, integrate the bound in \eqref{eq:A5gamma}
instead of taking its maximum along the segment. The even-parity case gives
\[
 \frac12\left|\psi_0\left(\frac{2+iu}{2}\right)
              -\psi_0\left(\frac{1/4+iu}{2}\right)\right|
 \le\frac12\int_{1/8}^{1}(r^{-2}+r^{-1})\,dr
 =\frac{7+\log8}{2}.
\]
The odd-parity segment is a translate by $1/2$ and gives a smaller bound.
For $\zeta$, the additional rational-function differences cost at most
$4+4/3+1/2+1=41/6$.
Using the affine bound in \eqref{eq:A5zeros}, we obtain
\begin{equation}\label{eq:A5left}
 |L'/L(1/4+iu)|\le\frac72\log(q^*(|u|+3))+42
 \le42\log(q^*(|u|+3)),
\end{equation}
since $28+2+(7+\log8)/2+41/6<42$ and
$7/2+42/\log3<42$. We retain the affine form.

Use the contour of Lemma~\ref{lem:A4}. For primitive non-principal characters,
\cite[Lemmas 12.6--12.7, p.~402]{MV07} supplies the admissible heights;
for $\zeta$ use \cite[Lemmas 12.1--12.2, pp.~397--398]{MV07} and conjugation.
Here no unspecified height-selection constant enters the numerical estimate.
Indeed, first fix $q^*$, $t$ and $v$. On an admissible horizontal edge
$\Im(s+z)=T_j$, uniformly for $1\le\sigma\le2$, its length is bounded,
$v^{\Re z}\le v^2$, $\widehat w(z)=O(|T_j-t|^{-2})$, and
$L'/L(s+z)=O(\log^2(q^*(|T_j|+3)))$. Thus the edge integral is
\[
 O_{q^*,t,v}\left(\frac{\log^2(q^*(|T_j|+3))}{|T_j-t|^2}\right)\longrightarrow0.
\]
The zero series is absolutely convergent by \eqref{eq:A5zeros} and
\eqref{eq:kernelsharp}, and the limiting vertical integral is absolutely
convergent by \eqref{eq:A5left}. Passing to the limit gives an exact contour
identity. Only then do we apply the following uniform numerical majorants;
there is no remaining horizontal-edge error or unspecified constant.

Use $C_w<8/3$ in \eqref{eq:kernelsharp}. The zero residues are at most
\[
 (2\log\mathcal Q^*+16)v^{1/2-\sigma}.
\]
For the vertical integral, the elementary convolution estimate is
\[
 \int_\R\frac{\log(q^*(|u|+3))}{1+(u-t)^2}\,du
 \le\pi\log\mathcal Q^*+\frac{3\pi}{2}\log2.
\]
This follows from $|u|+3\le(|t|+3)(1+|u-t|)$ and
$\log(1+|r|)\le\frac12\log2+\frac12\log(1+r^2)$, together with
$\int_\R\log(1+r^2)/(1+r^2)\,dr=2\pi\log2$.
The factor $1/(2\pi)$ in Mellin inversion therefore gives the bound
\[
 \frac83\left(\frac74\log\mathcal Q^*+21+\frac{21}{8}\log2\right)
 v^{1/4-\sigma}
 \le(5\log\mathcal Q^*+62)v^{1/4-\sigma},
\]
using $21+(21/8)\log2<23$. We do not replace $v^{1/4-\sigma}$
by $v^{1/2-\sigma}$.
For $\zeta$ with $|t|\ge1$, the pole residue costs only $v^{1-\sigma}$,
since the integral representation of the kernel gives
\[
 |\widehat w(1-\sigma-it)|
 \le\frac1{|1-\sigma-it|}\int_1^2 u^{1-\sigma}\,du\le1.
\]
Integration in $\sigma$ from $1$ to $2$ gives the first three terms in
\eqref{eq:A5}. This also justifies the boundary value at $\sigma=1$;
the logarithm is continued from the Euler-product half-plane.
The endpoint correction at $2+it$ is at most
$\sum_{n>v}n^{-2}\le2/v$.

For imprimitive characters subtract the missing Euler factors as in
Lemma~\ref{lem:A4}. Their first powers beyond $v$ cost at most $\log q/(v\log v)$.
Their higher-power tails are at most
\begin{equation}\label{eq:A5pptail}
 \sum_{p^k>v,\ k\ge2}\frac1{kp^k}\le\frac4{\sqrt v}.
\end{equation}
Indeed, primes $p>\sqrt v$ contribute at most
$\frac12\sum_{n>\sqrt v}(n(n-1))^{-1}\le1/\sqrt v$.
For $p\le\sqrt v$, start at the least $k_0$ with $p^{k_0}>v$;
the remaining geometric sum is at most $2\log p/(v\log v)$.
Its sum is at most $2.04/(\sqrt v\log v)$, and $1+2.04/\log2<4$.
Since $\mathcal Q^*\le\mathcal Q$, this proves \eqref{eq:A5}.
Finally $v^{-3/4}\le v^{-1/2}$ and
$7+78/\log3<80$ give \eqref{eq:A5compact}. $\qed$

\begin{proof}[Proof of Theorem~\ref{thm:N}(iii)]
Assume GRH, let $\ell\ge64$, and let
$q\le\ell$, $|t|\le x\ell$.
Put
\[
\mathcal Q_*:=\ell(x\ell+3),\qquad
b:=\log\mathcal Q_*,\qquad Y:=b^2.
\]
Then $e^\ell\le b\le2e^\ell$, $Y<x$, and
$\log\mathcal Q\le b$ for $\mathcal Q=q(|t|+3)$.
Write $P_v(\chi,t):=\sum_{p\le v}\chi(p)p^{-1-it}$.

First suppose that $\chi$ is non-principal or $|t|\ge1$.
Apply the sharper estimate \eqref{eq:A5} at $Y$ and $x$, and subtract
before removing prime powers. At $Y=b^2$ its first three terms are at most
\[
 \frac1\ell+\frac{8e^{-\ell}}\ell,
 \qquad\frac{5e^{-\ell/2}}{2\ell}+\frac{31e^{-3\ell/2}}\ell,
 \qquad\frac1{2\ell},
\]
respectively. The other three terms there are bounded by
$2e^{-2\ell}$, $(\log\ell)e^{-2\ell}/(2\ell)$, and $4e^{-\ell}$.
At the cutoff $x$, all six terms together are at most $100e^{-\ell}$:
using $b\le2e^\ell$ and $e^{-e^\ell/2}\le e^{-\ell}$, they are bounded
in order by $20e^{-\ell}$, $72e^{-\ell}$, $e^{-\ell}$, $2e^{-\ell}$,
$e^{-\ell}$, and $4e^{-\ell}$.
Consequently
\[
 |W_x-W_Y|\le\frac{3/2+R(\ell)}\ell<\frac{1.51}\ell,
\]
where
\[
 R(\ell):=8e^{-\ell}+\frac52e^{-\ell/2}+31e^{-3\ell/2}
 +2\ell e^{-2\ell}+\frac{\log\ell}{2}e^{-2\ell}+104\ell e^{-\ell}.
\]
Every term is decreasing for $\ell\ge64$, and
$R(64)<3.17\times10^{-14}<0.01$.

The prime powers at or below $Y$ cancel exactly between $W_x$ and
$W_Y$. Since the weight increases with its cutoff, the remaining
prime-power difference costs at most $4/\sqrt Y$ by
\eqref{eq:A5pptail}. For $v\ge286$, Rosser--Schoenfeld gives
\[
H(2v)-H(v)\le\frac{\log2}{\log v}+\frac1{\log^2v}.
\]
The two prime transition bands, together with $4/\sqrt Y$, therefore
cost at most
\[
\frac{\log2}{2\ell}+\frac1{4\ell^2}+6e^{-\ell}
<\frac1{2\ell}.
\]
After multiplication by $\ell$, the majorant is decreasing and is
less than $0.351$ at $64$. Since $1.51+0.351<1.9$, we have proved the uniform estimate
\begin{equation}\label{eq:A6}
|P_x(\chi,t)-P_Y(\chi,t)|\le\frac{1.9}{\ell}.
\end{equation}
Hence
\[
D_x(\chi,t)\ge H(x)-H(Y)-\frac{1.9}{\ell}.
\]
The Mertens constant cancels. Moreover,
\[
\log\log Y\le\log\ell+\log2
+\frac{(2\log\ell+1)e^{-\ell}}{\ell},
\]
because $b\le e^\ell+2\log\ell+1$.
The Mertens errors contribute at most
$e^{-2\ell}/2+1/(8\ell^2)$.
Their sum with the last cutoff correction is less than $0.1/\ell$:
after multiplication by $\ell$ the majorant is
\[
(2\log\ell+1)e^{-\ell}+\frac\ell2e^{-2\ell}+\frac1{8\ell},
\]
which is decreasing for $\ell\ge64$ and less than $0.002$ at $64$.
Consequently
\begin{equation}\label{eq:GRHexplicit}
D_x(\chi,t)\ge\ell-\log\ell-\log2-\frac2\ell.
\end{equation}

If $\chi=\chi_0$ and $|t|<1$, set $s:=\sigma_x+it$.
Euler--Maclaurin gives
\[
\zeta(s)=\frac1{s-1}+\frac12+\frac{s}{12}
-\frac{s(s+1)}2\int_1^\infty B_2(\{v\})v^{-s-2}\,dv.
\]
On $1\le\Re s\le2$, $|\Im s|\le1$, the integral term has modulus
at most $|s||s+1|/(12(\Re s+1))\le5/24$, while the real parts of
$1/(s-1)$ and $s/12$ are at least $0$ and $1/12$.
Thus $\Re\zeta(s)\ge3/8$ (excluding the pole itself).
The Euler-factor bound from Case L2 and Lemma~\ref{lem:A1} yield
\[
D_x(\chi_0,t)\ge\ell-\log(1+\log\ell)
-1.13-\log\bigl(8\zeta(2)/3\bigr)
\ge\ell-\log\ell-\log2.
\]
The final comparison has margin greater than $0.602$ at $64$, and
its derivative is positive. This is stronger than \eqref{eq:GRHexplicit}.

For $\ell\ge64$, let
$A_2:=\ell-\log\ell-\log2-2/\ell$.
Then $1\le A_2<\ell$: at $64$ it exceeds $59.11$, and its derivative
$1-1/\ell+2/\ell^2$ is positive. Thus the defining ranges $q\le A_2$, $|t|\le A_2x$
are covered, $A_2\in\cA(x)$, and $N_\lambda(x)\ge A_2$.
Together with the upper bound of (ii), this proves (iii).
\end{proof}

\address{Independent researcher}
\email{v@deltagray.com}
\end{document}